\documentclass{amsart}

\usepackage{hyperref}
\usepackage{amsmath,color,graphicx,amssymb,amscd}
\usepackage{degt} 
\usepackage{tikz}
\usepackage{tikz-cd}

\usepackage{booktabs}
\usepackage{comment} 
\usepackage[paper=portrait,pagesize]{typearea}
\usepackage{longtable}
\usepackage{array, boldline, makecell}

\newlength\mylength

\def\L{\mathbf{L}}
\def\K{\mathcal{K}}
\def\M{\mathcal{M}}

\def\Z{\mathbb{Z}}
\def\Q{\mathbb{Q}}

\def\discr{\operatorname{discr}}
\def\disc{\operatorname{disc}}

\def\Aut{\operatorname{Aut}}
\def\id{\operatorname{id}}

\def\ie{\emph{i.e\PERIOD}}
\def\eg{\emph{e.g\PERIOD}}
\def\iq.{\emph{i.q\PERIOD}}
\def\cf.{\emph{cf.}}
\def\viz.{\emph{viz\PERIOD}}
\def\vs.{\emph{vs\PERIOD}}

\begin{document}

\title[Real Structures in the Moduli of projective models of $K3$--surfaces]{Real Structures in the Moduli of projective models of K3-surfaces}

\author[\c{C}.~G\"{u}ne\c{s} ~Akta\c{s} ]{\c{C}\.{i}sem G\"{u}ne\c{s} Akta\c{s}}
\address{
 Department of Engineering Sciences, Abdullah G\"{u}l University\\
 38080, Kayseri, Turkey} \email{cisem.gunesaktas@agu.edu.tr}

\thanks{The author was supported by the \textit{T\"{U}B\.{I}TAK} grant $123$F$111$}



\date{}

\dedicatory{}

\begin{abstract}


This paper investigates when a complex family of $K3$-surfaces with prescribed
simple singularities contains real algebraic surfaces. 
We develop an explicit and uniform algorithm for detecting real representatives in any equisingular stratum, valid for all polarizations.
As a primary
application, we resolve the problem for spatial quartics, showing that
all real equisingular strata contain real surfaces except for three exceptional types, thereby completing the study of real representability in this case. 
Our method also revisits the analogous phenomenon in the realm of plane sextics, recovering the unique known real stratum of simple sextics without real representatives.


\end{abstract}

\maketitle


\section{Introduction}

The geometry of the equisingular strata of curves, surfaces, and
higher-dimensional varieties is one of the central problems of singularity
theory.  
In the case of surfaces of $K3$-type, this problem is particularly rich; thus,
$K3$-surfaces occupy a very special niche in the realm of surfaces: on the one
hand, they are sophisticated enough to pose interesting geometric problems,
and, on the other hand, they provide adequate tools for solving these problems.
It is commonly understood that, by means of the global Torelli theorem~\cite{K3} and the surjectivity of the
period map~\cite{periodmap}, any reasonable question about a $K3$-surface
can be restated in terms of its N\'{e}ron--Severi lattice, and the corresponding
arithmetical problem can be solved by means of Nikulin's \cite{Niku2} theory of
discriminant forms, extended by Miranda and Morrison \cite{MM1,MM2,MM3}.

\medskip
A projective model 
$
X \longrightarrow \mathbb{P}^{n+1}
$
of a $K3$-surface $X$ gives rise to a polarization, \ie, a class
$h\in H_{2}(X)$ of the hyperplane section, satisfying $h^{2}=2n$.
Irreducible smooth rational curves on $X$ are certain classes
$l\in\mathrm{NS}(X)\subset H_{2}(X)$ of square $(-2)$ intersecting $h$
in a prescribed way.  
The simplest case is that of exceptional divisors, $l\cdot h=0$, in which one
studies the equisingular deformation classification of singular models.
The case $h^{2}=0$ (elliptic $K3$-surfaces) was settled by
Shimada~\cite{Shimada.connEllK3}.

The case of plane sextic curves ($h^{2}=2$) was completed in
Akyol and Degtyarev~\cite{Alex2}, after numerous attempts to settle it by conventional, equation-based methods
(see, e.g., Artal et al.~\cite{Artal2001,Artal2007,Artal2002}  or
Oka and Pho~\cite{OkaPo2002,OkaPho2002b}).
That paper, builing  upon earlier work by
Degtyarev~\cite{Alex1,Degtyarev2008StableSym}, Persson~\cite{Persson},
Shimada~\cite{Shimada.Maximizing,Shimada:Zsplitting},
Urabe~\cite{Urabe1988,Urabe1989,Urabe3,Urabe2}, Yang~\cite{Yang.sextics},
gives a complete equisingular classification of simple sextics.

\medskip
The next case of spatial quartics ($h^{2}=4$) was treated classically by
Degtyarev~\cite{Alex94b} in the presence of a non-simple singularity,
\ie, when the quartic is not a $K3$--surface; already then it became apparent
that extending these methods to quartic surfaces with only simple
singularities was hopeless.  
The $K3$-theoretic approach, pioneered by Urabe~\cite{Urabe3,Urabe2} and
Yang~\cite{Yang}, was adopted by G\"{u}ne\c{s}~Akta\c{s}~\cite{Cisem1,aktacs2019real,Cisem2024}, who completed the
classification of simple quartics in the spirit of Akyol and Degtyarev~\cite{Alex2}.

Traditionally, in algebraic geometry one studies either singular (or otherwise enhanced, \eg, with many lines, conics) complex objects (curves, surfaces, \emph{etc.} ) or smooth real ones;  these two tasks are typically of comparable
difficulty.  
Combining the two perspectives is considerably more difficult, and
despite numerous attempts and an extensive literature, only a few sporadic
results are known. For example, speaking about deformation classification problems in the \emph{real} case, the work on quartic spectrahedra by Degtyarev and Itenberg~\cite{AI} and Ottem \emph{et al.}~\cite{Sturmfels} deals with singular \emph{real} quartics. Other special classes of singular real quartic surfaces were studied in Krasnov~\cite{Krasnov2018,Krasnov2019,Krasnov2020a,Krasnov2020b}. Also worth mentioning are other polarizations, namely, sextics in $\mathbb{P}^4$, in Krasnov~\cite{Krasnov2011} and more general $K3$-surfaces with Klein group actions in Degtyarev \emph{et al.}~\cite{Degtyarev:finiteness}. Topologically, the ultimate goal is the classification of singular real curves (surfaces, etc.) up to the so-called rigid isotopy, \ie, equivariant equisingular deformation. Despite its natural
formulation, this problem has not yet been resolved as the known approach needs the fundamental polyhedra of certain hyperbolic groups generated by reflections and, since usually these polyhedra are infinite, their computation via Vinberg \cite{Vinberg1972} gets out of hand. In this paper, we make an attempt to bridge this gap from a slightly different perspective by focusing on the interaction
between singularities and real structures, namely, we address the question: when does a real equisingular stratum in the moduli space of complex singular models
contain a real algebraic variety?
\subsection{Principal results}
We study all models of $K3$-surfaces of a certain fixed \emph{kind} (see 
\autoref{projective.models.of.K3surfaces}).
Denote by $\M$ the space of all models $f\colon X\rightarrow \mathbb{P}^n$; it is divided into equisingular strata $\M(S)$ where $S$ is the set of simple singularities.  Each stratum $\M(S)$  splits further into its connected components corresponding to equisingular deformation classes.

A standard \emph{real structure} (\emph{i.e.}, an antiholomorphic involution) $\operatorname{conj}\colon \mathbb{P}^n\rightarrow\mathbb{P}^n$ induces a real structure $\operatorname{c}: \M\rightarrow\M$ by sending $f\colon X\rightarrow \mathbb{P}^n$ to $\operatorname{conj}\circ f\colon \bar{X}\rightarrow \mathbb{P}^n$. This real structure $\operatorname{c}$ depends on the choice of $\operatorname{conj}$; however, 
the induced action on the set of the connected components of the equisingular strata is well defined. A connected component $\mathcal{D}\subset\mathcal{M}(S)$ is called \textit{real} if $\operatorname{c}(\mathcal{D})=\mathcal{D}$. Clearly, each stratum $\mathcal{M}(S)$ consists of real and pairs of complex conjugate components; this classification of components is given in~\cite{Alex2} for sextics  and in~\cite{Cisem1,Cisem2024} for (nonspecial) quartics.

Although it is quite common that a real variety may have no real points, very few examples of equsingular deformation classes with this property are known. Clearly, any class $\mathcal{D}\subset\M(S)$ containing a real model is real; however, the converse is not true. The fisrt known example of a 
\emph{real} equisingular stratum containing no real curve is the stratum $\M(\mathbf{A}_7\oplus\mathbf{A}_6\oplus\mathbf{A}_5)$ of the space of sextics found in \cite{Alex2}. This phenomenon raises the analogous question for spatial quartic 
surfaces which, until now, remained open except in the 
nonspecial case treated in~\cite{aktacs2019real}.

The present paper provides a complete solution to this problem.
Our principle result, Theorem~\ref{principal.result}, shows that among all equisingular strata of simple quartic surfaces, exactly
three real strata fail to contain real quartic surfaces. This settles the real representability question for simple quartics, extending the earlier results for nonspecial quartics obtained in~\cite{aktacs2019real}.  In particular, we identify a new
exceptional example in the special locus,
$\mathbf{A}_{7}\oplus\mathbf{A}_{5}\oplus \mathbf{A}_{3}\oplus \mathbf{A}_{2}\oplus \mathbf{A}_{1}$, complementing the two
nonspecial examples discovered in~\cite{aktacs2019real}.

\begin{theorem}[Degtyarev, G\"{u}ne\c{s} Akta\c{s}]\label{principal.result}
	With the exception of (one of the strata for each of) the following
	sets of singularities:
	\begin{align*}
	\mathbf{A}_7\oplus\mathbf{A}_6\oplus\mathbf{A}_3\oplus\mathbf{A}_2 \text{(nonspecial)},\\ \mathbf{D}_7\oplus\mathbf{A}_6\oplus\mathbf{A}_3\oplus\mathbf{A}_2\text{(nonspecial)},\\ \mathbf{A}_7\oplus\mathbf{A}_5\oplus\mathbf{A}_3\oplus\mathbf{A}_2\oplus\mathbf{A}_1\text{(special)},
	\end{align*} 
	each real equisingular stratum of simple quartic surfaces
	$X\subset\mathbb{P}^{3}$ has a real representative.
\end{theorem}

Our approach is based entirely on lattice theory whose tools yield an effective and uniform algorithm for detecting real 
representatives in any equisingular stratum of a simple projective $K3$--model.  
We apply this algorithm to two fundamental families:
simple quartic surfaces ($h^{2}=4$) and simple sextic curves ($h^{2}=2$).
In the quartic case, the algorithm determines a full account of which real strata admit real representatives. In the
sextic case, our method recovers the exceptional example of
Akyol and Degtyarev~\cite{Alex2} and provides a conceptually simpler proof based on
the lattice–-theoretic insight. Our result is obtained by implementing the algorithm decribed in \S\ref{section.algorithm} in \texttt{GAP}~\cite{GAP} as the number of classes (about 12000 ) is far beyond what can be handled by hand .

\subsection{Contents of the paper}
In $\S2$, we recall the necessary background on integral lattices, discriminant forms, and their extensions, following Nikulin’s~\cite{Niku2} foundational work. 
In $\S3$, we introduce projective models of $K3$-surfaces and interpret them in purely lattice-theoretic terms, introducing the lattice types and homological types associated with simple singularities. The section concludes with the arithmetical reduction (Theorem~\ref{def.class}), which translates equisingular deformation classes into oriented homological types. In $\S4$, we discuss real structures on projective models and obtain a criterion for the existence of real representatives in terms of involutive skew–-autoisometries (Theorem~\ref{thm:realization_general}).
In $\S5$, we outline two key tools for detecting such involutions: the perturbation method for propagating reality from maximizing types, and the reflection method searching for real structures arising from (negative) reflections in the transcendental lattice.
In $\S6$, we formulate the main lattice criterion (Theorem~\ref{thm:Aktas2019}) describing when a reflection can be extended to a global skew–-autoisometry of the 
$K3$--lattice, together with a complete characterization in the submaximal case with the total Milnor number $\mu (X)=18$ (Corollary~\ref{cor:mu18}).
In $\S7$, we assemble these ingredients into a complete algorithm for detecting real representatives in any equisingular stratum of projective models of any kind. The algorithm begins with perturbation reduction, and when this fails, proceeds to a systematic search for reflections satisfying the discriminant and embedding conditions required by Theorem~\ref{thm:Aktas2019}. All steps are fully implemented in \texttt{GAP}, and the section explains how they combine to yield definitive conclusions for all lattice types.
In $\S7$, we apply the algorithm to the two fundamental projective models:
simple quartic surfaces, where we prove that all real strata contain real quartics except for three explicitly listed exceptional types;
and simple sextic curves, where we recover the known known real stratum that admits no real representatives via the new lattice-theoretic mechanism.

\subsection{Acknowledgments}
I am grateful to Alexander Degtyarev for a number of comments, suggestions and motivating discussions which significantly contributed to improving the quality of this paper.


\section{Integral lattices}
We recall briefly a few notions and known results concerning integral lattices, their discriminant forms and extensions. The principal reference is \cite{Niku2}.
\subsection{Finite quadratic forms}

A \emph{finite quadratic form} is a finite abelian group $\mathcal{L}$ equipped with a map $q\colon \mathcal{L}\rightarrow\mathbb{Q}/2\mathbb{Z}$ quadratic in the sense that 
$$q(x+y)=q(x)+q(y)+2b(x,y),\quad q(nx)=n^2q(x),\quad x,y\in\mathcal{L},\;n\in \Z,$$ 
where $b\colon \mathcal{L}\otimes\mathcal{L}\rightarrow\mathbb{Q}/\mathbb{Z}$ is a symmetric bilinear form (which is determined by $q$) and $2\colon \mathbb{Q}/\mathbb{Z}\rightarrow\mathbb{Q}/2\mathbb{Z}$ is the natural isomorphism. We abriviate  $x^2:=q(x)$ and $x\cdot y:=b(x,y)$. 

Each finite quadratic form can be decomposed into the orthogonal direct sum $\mathcal{L}=\bigoplus_p\mathcal{L}_{p}$ of its $p$-primary components $\mathcal{L}_{p}:=\mathcal{L}\otimes \mathbb{Z}_p$, where the summation runs over all primes $p$. The \emph{length} $\ell( \mathcal{L})$ is the minimal  number of generators of $\mathcal{L}$; we put $\ell_p(\mathcal{L}):=\ell(\mathcal{L}_{p})$.
A finite quadratic form $\mathcal{L}$ is called \emph{even} if $x^2=0\bmod \Z$ for each element $x\in \mathcal{L}_{2}$ of order $2$; it is called \emph{odd} otherwise.

A finite quadratic form is \emph{nondegenerate} if the homomorphism
\begin{align*}
\mathcal{L}\rightarrow \operatorname{Hom}(\mathcal{L},\mathbb{Q}/\mathbb{Z}),\quad x\mapsto(y\mapsto x\cdot y)
\end{align*}
is an isomorphism. We denote by $\Aut (\mathcal{L})$ the group of automorphisms of $\mathcal{L}$ preserving the form $q$. A subgroup $\K\subset \mathcal{L}$ is called \emph{isotropic} if the restriction of the quadratic form $q$ on $\mathcal{L}$  to $\K$ is identically zero. If this is case $\K^{\bot}/\K$ also inherits from $\mathcal{L}$ a nondegenerate quadratic form.

For a fraction $\frac{m}{n}\in \Q/ 2\Z $, with $(m,n)=1$  such that $mn=0\bmod 2$, we denote by $\langle\frac{m}{n}\rangle$  the nondegenerate finite quadratic form on $\Z/n\Z$ sending the generator to $\frac{m}{n}$, \emph{i.e.} , $\alpha^2=\frac{m}{n} \bmod 2\Z$ for a generator $\alpha$. 
For an integer $k\ge 1$, let $\mathcal{U}_n$ and $\mathcal{V}_n$ be the length $2$ forms on $(\Z/n\Z)^2$, defined by the matrices
\begin{align*}
\mathcal{U}_n:=\left[
\begin{array}{cc}
0 & 1/n \\
1/n & 0 
\end{array}
\right],\quad \mathcal{ V}_n:=\left[
\begin{array}{cc}
2/n&  1/n\\
1/n & 2/n
\end{array}\right], \quad \mbox{where $n=2^k$}.
\end{align*}
Nikulin \cite{Niku2} proved that, a nondegenerate finite quadratic form splits into an orthogonal direct sum of cyclic forms $\langle \frac{m}{n}\rangle$ (defined on the cyclic group  $\Z/n\Z$) and length 2 blocks $\mathcal{U}_n$, $\mathcal{V}_n$. Unless the $2$-torsion  consists of the summands of length 2, we describe nondegenerate finite quadratic forms by expressions of the form $\langle q_1\rangle\ldots \langle q_r\rangle$, where $q_i=\frac{n_i}{m_i}\in \Q$ as above; the group is generated by pairwise orthogonal elements $\alpha_1,\ldots\alpha_n$ (numbered in the order of appearance) so that $\alpha_i^2=\frac{m_i}{n_i} \bmod 2\Z$ and order of $\alpha_i$ is $n_i$.

Given a prime $p$, the \emph{determinant} $\det_p\mathcal{L}$ of a nondegenerate finite quadratic form is the determinant of the matrix of the form on the $p$-group $\mathcal{L}_{p}$ in any minimal basis. Due to \cite{MM3}, one has  $\det_p\mathcal{L}$ has the form  $u/|{\mathcal{L}_{p}}|$, where $u\in\Z_p^{\times}$; the unit $u$ is well defined  modulo $(\Z_p^{\times})^2$ unless $p=2$ and $\mathcal{L}_2$ is odd; in the latter case, $\det_2\mathcal{L}$ is well defined  modulo the subgroup generated by $(\Z_2^{\times})^2$ and $5$.

\subsection{Integral lattices and discriminant forms}
An \emph{(integral) lattice} is a finitely generated free abelian group $L$ equipped with a symmetric bilinear form $b\colon L\otimes L\rightarrow \mathbb{Z}$. Whenever the form is fixed, we use the abbreviation $x^2:=b(x,x)$ and $x\cdot y:=b(x,y)$. A lattice $L$ is called \emph{even} if $x^2 =0\mod 2$ for all $x\in L$; it is called \emph{odd} otherwise. The \emph{determinant} $\det L \in \Z$ is the determinant of the Gram matrix of $b$ in any integral basis of $L$. 
A lattice $L$ is called \emph{unimodular} if $\det L=\pm 1$; it is called \emph{nondegenerate} if $\det L \neq 0$, or equivalently, the \emph{kernel}
\begin{align*}
\operatorname{ker}L=L^{\bot} :=\{x\in L \mid\text{ $x\cdot y= 0$ for all $y\in L$}\}
\end{align*}
is trivial.

The \emph{signature} of a nondegenerate lattice $L$ is the pair $(\sigma_+,\sigma_-)$ of its inertia indices. A nondegenerate lattice is called  \emph{hyperbolic} if $\sigma_+L=1$. 
Given a lattice $L$, the bilinear form on $L$ can be extended by linearity to a $\mathbb{Q}$-valued bilinear form on $L\otimes\mathbb{Q}$. If $L$ is nondegenerate, then we have  canonical inclusion
\begin{align*}
L\subset L^{\vee}:=\operatorname {Hom}(L,\Z)=\{x \in L\otimes\Q \mid \text{$x\cdot y \in \Z$ for all $y \in L$}\}
\end{align*}
The finite quotient group $\discr L :=L^{\vee}/L$ of order $\mathopen|{\det L}\mathclose|$ is called the \emph{discriminant group}  of $L$. In particular, $L$ is unimodular if and only if $\discr L=0$, \emph{i.e.}, $L=L^{\vee}$

The discriminant group inherits from $L\otimes\mathbb{Q}$ a nondegenerate symmetric bilinear form
\begin{align*}
b\colon \discr L\otimes \discr L\rightarrow \mathbb{Q}/\mathbb{Z},\quad (x\bmod L)\otimes(y\bmod L)\mapsto(x\cdot y)\bmod \Z,
\end{align*}
and if $L$ is even, its quadratic extension
\begin{align*}
q\colon \discr L\rightarrow \mathbb{Q}/2\Z, \quad (x \bmod L) \mapsto x^2 \bmod2\Z,
\end{align*}
called, respectively,  the \emph{discriminant bilinear form} and \emph{discriminant quadratic form}. Note that the discriminant group of an even lattice is a finite quadratic form. We use the notation $\discr_p L$ for the $p$-primary part of $\discr L$.
When speaking about the discriminant groups and  their (anti-)isomorphisms, these forms are always taken into account.

Nikulin~\cite{Niku2} states that two nondegenerate even lattices $L',L''$ are in the same \emph{genus} if and only if one has $\operatorname{rk} L'=\operatorname{rk} L''$, $\sigma L'=\sigma L''$ and $\discr L'\cong\discr L''$. Each genus consists of finitely many isomorphism classes(see \cite{Serre.Cours}).

An isometry  $\psi \colon L \rightarrow L'$ between two lattices is a group homomorphism respecting the bilinear forms; obviously one always has $\operatorname{Ker} \psi \subset \operatorname{Ker} L$. The group of autoisometries of a nondegenerate lattice $L$ is denoted by $O(L)$. The action of $O(L)$ extends to $L\otimes\mathbb{Q}$ by linearity, and the latter action descents to  $\discr L$. Therefore, there is a natural homomorphism $O(L)\rightarrow \Aut(\discr L)$ where $\Aut(\discr L)$ denotes the group of automorphisms of $\discr L$ preserving the discriminant form $q$ on $\discr L$. In general this map is neither one-to-one nor onto; however, 
without any confusion we freely apply autoisometries $g\in O(L)$ to objects in $\discr L$. Obviously one has $\Aut(\discr L)=\prod_p \Aut(\discr_p L)$ where the product runs over all primes. The restriction of $d$ to $p$-primary components are denoted by $d_p\colon O(L)\rightarrow \Aut(\discr_p L)$.

Given an integer $d\ge 2$, a $2d$-\emph{polarized lattice} is a nondegenerate hyperbolic lattice $L$ equipped with a distinguished vector $h\in L$ such that $h^2=2d$.
The group of autoisometries of $L$ preserving $h$ is denoted by $O_h(L)$. 

The orthogonal projection establishes a linear isomorphism between any two maximal positive definite  subspaces in $L\otimes\mathbb{R}$, thus providing a way for comparing orientations. A coherent choice of orientations of all maximal positive definite subspaces is called a \textit{positive sign structure} on $L$. We denote by $O^+(L)\subset O(L)$ the subgroup consisting of the autoisometries preserving a positive sign structure. Either one has $O^{+}(L)=O(L)$ or $O(L)^+$ is a subgroup of $O(L)$ of index $2$. In the latter case, each element of $O(L)\smallsetminus O^+(L)$ is called a \emph{skew-autoisometry} of L,\emph{ i.e.}, the autoisometries of $L$ that reverse the positive sign structure.

Of special importance are so called reflections of $L$: given a nonzero vector $a\in L$, the \emph{reflection} defined by $a$ is the automorphism 
\[t_a\colon L\rightarrow L,\quad x\mapsto \frac{2(a\cdot x)}{a^2}a .\]
It is well defined if and only if $(2a/a^2)\in L^\vee$. Note that $t_a$ is an involution. Each image $d_p(t_a)$ is also a reflection and if $a^2=\pm1$ or $a^2=\pm2$, then the induced automorphism $d(t_a)$ of the discriminant group is the identity and $t_a$ extends to any lattice containing $L$.

\subsection{Root Lattices}\label{root.lattices}
A \emph{root} in an even lattice is a vector of square $(-2)$. A \emph{root lattice} is a negative definite lattice generated by its roots. Each root lattice splits uniquely (up to order of summands) into orthogonal sum of irreducible root lattices which are of type $\textbf{A}_n$, $n\geq1$, $\textbf{D}_n$, $n\geq4$, or $\textbf{E}_n$, $n=6,7,8$. For further details on irreducible root systems, see \cite{Bour}.

Given a root lattice $S$, we have $O(S)=R(S)\rtimes\operatorname{Sym}(\Gamma)$, where $R(S)\subset O(S)$ is the group generated by reflections $t_u$, $u\in S$, $u^2=-2$ and $\operatorname{Sym}(\Gamma)$ is the group of  symmetries of the Dynkin graph $\Gamma_S:=\Gamma$. Let $\operatorname{Sym}'(\Gamma)$ be the group of symmetries of $\textbf{E}_8$-type components. Then  one has 
\begin{align*}
\operatorname{Ker}[d\colon O(S)\rightarrow\Aut (\discr S)] = R(S)\rtimes\operatorname{Sym}'(\Gamma). 
\end{align*}
Hence, $d$ admits an isomorphism
\begin{align}\label{Sym0}
    \operatorname{Im}d \cong\operatorname{Sym}_0(\Gamma)\subset\operatorname{Sym}(\Gamma)\subset O(S)
\end{align}
where $\operatorname{Sym}_0(\Gamma)$ is the group of permutations acting identically on the union of $\textbf{E}_8$-type components.
\subsection{Lattice extensions}\label{lattice.extensions}
From now on unless specified otherwise all lattices considered are even and nondegenerate. An extension of an even lattice $S$ is an even lattice $L$ containing $S$. Two extensions $L',L''\supset S$ are called \emph{isomorphic} if there is a bijective isometry $L'\rightarrow L''$ preserving $S$, in particular, if the isomorphism  $L'\rightarrow L''$ is identical on $S$, the extensions $L'$ and $L''$ are called \emph{strictly isomorphic}. More generally, one can also fix a subgroup $G\in O(S)$ and speak about $G$-\emph{isomorphisms} of the extensions, \emph{i.e.}, bijective isometries whose restriction to $S$ is in $G$.

The two extreme cases are \emph{finite index extensions}, \emph{i.e.}, $L$ contains $S$ as a subgroup of finite index and \emph{primitive extensions}, \emph{i.e.}, $L/S$ is torsion free. The general case $L\supset S$ splits into the finite index extension $\tilde{S}\supset S$ and primitive extension $L\supset\tilde{S}$, where 
$$\tilde{S}:=\{x\in L\mid nx\in S\;\mbox{for some $n\in\Z$}\}$$
is the \emph{primitive hull} of $S$ in $L$. 

Any extension $L\supset S$ of finite index admits a unique embedding $L\subset S\otimes \Q$. Since $S$ is nondegenerate, we have $L\subset S^\vee$, and hence the natural inclusions
\begin{align*}
S\subset L\subset L^{\vee}\subset S^{\vee}.
\end{align*}  
The subgroup $\K:=L/S\subset S^{\vee}/S=\discr S$ is called the \emph{kernel} of the finite index extension $L\supset S$. This subgroup $\K$ is \emph{isotropic} (since $L$ is an even integral lattice), \emph{i.e.}, the restriction to $\K$ of the quadratic form $q\colon \discr S\rightarrow \mathbb{Q}/2\Z $ is identically zero. Conversely, if $\K\subset \discr S$ is isotropic, the lattice
$$L:=\{x\in S\otimes \mathbb{Q} \mid x\bmod S\in \K\}$$
is an extension of $S$ and we say that $L$ is the extension of $S$ by $\K$. This can be summarized in the following statement.
\begin{theorem}[Nikulin~\cite{Niku2}]\label{L-K}
	Let $S$ be a nondegenerate even lattice, and fix a subgroup $G\subset
	O(S)$. The map
	\begin{align*}
	(L\supset S)\mapsto \mathcal{K}:=L/S \subset \discr S
	\end{align*} 
	establishes a one-to-one correspondence between the set of
	$G$-isomorphism classes of finite index extensions $L\supset S$
	and the set of $G$-orbits of isotropic subgroups
	$\mathcal{K}\subset \discr S$. Under this correspondence one
	has $\discr L=\mathcal{K}^{\bot}/\mathcal{K}$. Furthermore, an autoisometry of $S$ extends to a finite index extension $L\supset S$ if and only if it preserves $\K$.
	
\end{theorem}

An extension $L\supset S$ is called \emph{primitive} if $L/S$ is torsion free. Such extensions are studied by fixing (the isomorphism class of) the orthogonal complement $T:=S^{\bot}\in L$. Then $L$ is a finite index extension of $S\oplus T$, in which $T$ is also primitive, and by Theorem \ref{L-K}, it is described by an isotropic subgroup
\begin{equation*}
\K\subset \discr (S\oplus T)=\discr S\oplus \discr T,
\end{equation*}
and the primitivity of $S$ and $T$ in $L$ implies that
\begin{equation*}
\K\cap \discr S= \K\cap \discr T= 0.
\end{equation*}
In other words, $\K$ is the graph of a certain monomorphism $\psi\colon \mathcal{ D}\rightarrow  \discr T$, where $\mathcal{ D}\subset \discr S$.  Since $\K$ is isotropic, $\psi$ is an anti-isometry. 

If $L\supset S\oplus T$ above is unimodular, $\discr L=\mathcal{K}^{\bot}/\mathcal{K}=0$, \emph{i.e.}, $\mathcal{K}^{\bot}=\mathcal{K}$. Then $|\mathcal{K}|^2= |\discr S|\cdot|\discr T|$ which implies that $|\mathcal{K}|= |\discr S|=|\discr S|$  and $\psi$ above is an anti-isomorphism $\discr S\rightarrow\discr T$. Since also $\sigma_{\pm}T=\sigma_{\pm}L-\sigma_{\pm}S$, it follows that the genus $g(T)$ is determined by the genera $g(S)$ and $g(L)$; we will denote this common genus by $g(S^{\bot}_L)$ (We emphasize that $g(S^{\bot}_L)$   may be empty, \emph{cf.} Theorem \ref{th.N.existence} below). If $L$ is also indefinite, it is unique in its genus (see, \emph{e.g.}, Siegel~\cite{SiegelI,SiegelII,SiegelIII}). Thus, we have the following statement (\cf. Nikulin~\cite{Niku2}).
\begin{theorem}\label{bicoset}
	Given  a subgroup $G\subset O(S)$ and unimodular even indefinite lattice $L$, a primitive isometry  $S \into L$ give rise to a bijective isometry $\psi\colon \discr S\rightarrow -\discr S^\bot$ and $G$-isomorphism classes of a primitive isometries  $S\into L$ are in canonical bijection with the following sets of data:
	\begin{enumerate}
		\item an even lattice (isomorphism class) $T\in g(S^{\bot}_L)$, and
		\item a bi-coset in $G\backslash \Aut (\discr T)/O(T)$.
	\end{enumerate}
\end{theorem}

In particular, the extension $L\supset S$ exists if and only if the genus $g(S^{\bot}_L)$ is nonempty.

From now on, we fix the notation $\L:=3\textbf{U}\oplus2\textbf{E}_8$ where $\textbf{U}$ stands for the \emph{hyperbolic plane}, the lattice generated by a pair of vectors $u,v$ (refered as the \emph{standard basis} of $\textbf{U}$) with $u^2=v^2=0$ and $u\cdot v=1$. Note that $3\textbf{U}\oplus2\textbf{E}_8$ is the unique even unimodular lattice of rank $22$ and signature $(3,19)$. We are concerned about this lattice as it is the intersection index form of a $K3$-surface $X$, \ie, $H_2(X;\Z)\cong \L$.   
We are interested in the primitive embeddings to $\L$. The following theorem giving a criterion for $g(S^{\bot}_\textbf{L})\neq\emptyset$ is a combination of the above observation and Nikulin's existence theorem~\cite{Niku2} applied to the genus $g(T)$.
\begin{theorem}[Nikulin~\cite{Niku2}]\label{th.N.existence}
Given a nondegenerate even lattice~$S$, a primitive extension $\textbf{L}\supset S$ exists
if and only if
the following conditions hold
\begin{enumerate}
\item $\sigma_+ S\leq 3$, $\sigma_-S\leq 19$ and $\ell(\mathcal{S})\leq 22- \operatorname{rk} S$, where $\mathcal{S}=\disc S$;
\item one has  $|{\mathcal{S}}|\det_p (\mathcal{S}) =(-1)^{\sigma_+S-1} \bmod (\Z_p^{\times})^2$ for each  odd prime $p$ such that  $\ell_p(\mathcal{S})= 22- \operatorname{rk} S$;
\item If $\ell_2(\mathcal{S})= 22- \operatorname{rk} S$, and $\mathcal{S}_2$ is even  then $|{\mathcal{S}}|\det_2 (\mathcal{S}) =\pm 1 \bmod (\Z_2^{\times})^2$.
\end{enumerate}
\end{theorem}

\section{K3-Surfaces}\label{projective.models.of.K3surfaces}
In this section, we give a brief account of the theory of K3-surfaces and their projective models; for more details and further references, we address the reader to  \cite{Huybrechts:K3}.

\subsection{Topological structure of K3-Surfaces}\label{lattice.structure.of.K3surfaces}
A $K3$-surface over $\C$ is a simply connected, compact complex surface whose canonical bundle is trivial. All $K3$-surfaces are K\"{a}hler, hence in particular a compact complex surface admitting a rich supply of algebraic models. 
Given a $K3$-surface $X$, one has 
$$  H_2(X;\Z)\cong \mathbf{L}=3\textbf{U}\oplus 2\textbf{E}_8.$$
This lattice $\mathbf{L}$ is unique up to isomorphism and is the only even unimodular lattice of signature $(3,19)$, see \S\ref{lattice.extensions}. For every projective $K3$-surface, the N\'{e}ron-Severi lattice
\[
\mathrm{NS}(X) \;:=\; H^{1,1}(X)\,\cap\, H^{2}(X;\mathbb{Z})
\]
is a primitive sublattice of $\mathbf{L}$ of signature $(1,\rho(X)-1)$, 
where $\rho(X)=\operatorname{rk}\mathrm{NS}(X)\le 20$.
A $K3$-surface $X$ is projective if and only if its N\'{e}ron--Severi lattice 
$\mathrm{NS}(X)$ contains a class $h$ with $h^{2}>0$, 
or equivalently, if $X$ admits an ample line bundle. 
In this case, the intersection form on $\mathrm{NS}(X)$ is hyperbolic.

Even though $K3$-surfaces are those that are given in some $\mathbb{P}^n$ by a system of polynomial equations, these equations almost never enter the picture: by means of such fundamental results as the global Torelli theorem~\cite{K3}, the surjectivity of period map~\cite{periodmap} and the results of Saint-Donat~\cite{Donat}, we identify a $K3$-surface $X$ with its polarized N\'{e}ron-Severi lattice $NS(X)\ni h$ and study the latter by purely arithmetical means, see \S\ref{lattice.extensions} and \S\ref{arithmetical.reduction}.

\subsection{Projective models of K3-Surfaces}\label{projective.models.of.K3surfaces}

Let $X$ be a smooth projective $K3$-surface and 
$h\in\mathrm{NS}(X)$ a primitive ample divisor with $h^{2}=2d>0$. 
The complete linear system $|h|$ defines a morphism
\[
f_{h}: X \longrightarrow \mathbb{P}^{\,d+1}.
\]
This morphism is called a \emph{projective model of degree~$2d$} of the $K3$-surface~$X$ and the class $h$ is referred to as \emph{polarization}. According to \cite{Donat}, such a model can be \emph{birational} (generically one-to-one) or \emph{hyperelliptic} (generically two-to-one) where its image is a normal projective surface of degree $2d$ or $d$, respectively. (If $d = 1$, all models are hyperelliptic, and we study the ramification locus, which is a plane sextic curve.) Up to projective equivalence, the low-degree projective models are given as follows:

\vspace{0.3cm}
	
	\makebox[\textwidth]
	{Birational projective models of $K3$-surfaces}
	\begin{tabular}{ccll}
		\toprule
		$h^{2}=2d$ & Morphism $f_{h}$ & Image $f_{h}(X)\subset\mathbb{P}^{\,d+1}$ & Description\\
		\midrule
		$4$ & $f_{h}:X\to\mathbb{P}^{3}$ & Quartic surface & Spatial model.\\
		$6$ & $f_{h}:X\to\mathbb{P}^{4}$ & Intersection of a quadric and a cubic & Sextic model.\\
		$8$ & $f_{h}:X\to\mathbb{P}^{5}$ & \parbox{6cm}{Octic surface (embedded case: intersection of three quadrics)} & Octic model.\\
	
		\bottomrule
	\end{tabular}

\makebox[\textwidth]{Hyperelliptic projective models of $K3$-surfaces}
\begin{tabular}{ccll}
		\toprule
		$h^{2}=2d$ & Morphism $f_{h}$ & The map $f_{h}(X)\rightarrow\mathbb{P}^{\,d+1}$ & Description\\
		\midrule
		$2$ & $f_{h}:X\to\mathbb{P}^{2}$ & Double cover branched along a sextic curve & Planar model.\\
		$4$ & $f_{h}:X\to Q\subset\mathbb{P}^{3}$ & \parbox{6cm}{Double cover of smooth quadric $Q$  branched along $(4,4)$ curve} &\parbox{2.7cm}{ Hyperelliptic quartic model.}\\
		\bottomrule
\end{tabular}

 Given a projective model $f_h\colon X\rightarrow \mathbb{P}^{\,d+1}$, the intersection lattice $L_X:=H_2(X;\Z)$ is of the form
\begin{equation*}
L_X=H_2(X;\Z)\cong \L.
\end{equation*}
For each simple singular point $p$ of $X$, the components of the exceptional divisor over $p$ span a root lattice in $\mathbf{L}_X$. The orthogonal sum of these sublattices, denoted by $S_X$, is identified with the set of singularities of $X$. Recall that the types of individual singular points are uniquely recovered from $S_X$, see \ref{root.lattices}. In what follows we identify homology and cohomology of $\tilde{X}$ \emph{via} Poicar\'{e} duality and introduce the following vectors and sublattices:
\begin{itemize}
	\item $S_X\subset L_X$: the sublattice generated by the curves contracted by  $f_h$;
	\item $S_{X,h}:=S_X\oplus \mathbb{Z}h_X\subset L_X$ where $h_X=h\in NS(X)$ is the class of the pull-back of a generic plane section of $X$;
	\item $\tilde{S}_X \subset \tilde{S}_{X,h}\subset L_X$: the primitive hulls of $S_X$ and $S_{X,h}$, respectively, \textit{i.e}, $\tilde{S}_X:=(S_X\otimes\mathbb{Q})\cap L_X$ and  $\tilde{S}_{X,h}:=(S_{X,h}\otimes\mathbb{Q})\cap L_X$;
	\item $\omega_X \subset L_X\otimes \mathbb{R}$: the oriented $2$-subspace spanned by the real and imaginary parts of the class of a holomorphic $2$-form on $X$ (the \emph{period} of $X$).
\end{itemize}
Recall that all singularities are \emph{simple}
and, hence, $S_X$ is a \emph{root} lattice, \emph{i.e.}, a negative definite lattice generated by vectors of square $(-2)$ (\emph{roots}). Note that $\omega_X$ is positive definite and orthogonal to $h_X$; in particular $\omega_X\in \tilde{S}_X^\bot\otimes \R$.
The rank $\operatorname{rk}(S_X)$ equals the total Milnor number $\mu(X)$. Since $S_X\subset \mathbf{L}$ is negative definite and $\sigma_-(\mathbf{L})=19$, one has $\mu(X)\leq 19$ (see \cite{Urabe2}, \emph{cf.}, \cite{Persson}). If $\mu(X)=19$, the the projective model is called  \textit{maximizing}. The triple $( h_X\in \tilde{S}_{X,h}\subset L_X)$ is called the \emph{homological type} of the projective model $f_h\colon X\rightarrow \mathbb{P}^{\,d+1}$. 
\subsection{Lattice--theoretic interpretation.}
Each projective model listed in \S\ref{projective.models.of.K3surfaces} corresponds to a 
polarized $K3$-surface $(X,h)$ with primitive ample class $h\in\mathrm{NS}(X)$ satisfying $h^{2}=2d$. 
The lattice $\mathrm{NS}(X)$ carries all numerical data of the embedding 
$f_{h}:X\to\mathbb{P}^{\,d+1}$: it determines both the singularity configuration 
(the sublattice $S\subset\mathrm{NS}(X)$ generated by classes of exceptional divisors) 
and the polarization $h$.  

The distinction between birational and hyperelliptic projective models has a simple 
lattice interpretation: it depends on the existence of isotropic or ``bad'' vectors 
relative to the polarization $h$.  
If $\mathrm{NS}(X)$ contains a class $e$ with
\[
e^{2}=0,\qquad e\!\cdot\!h=2,
\]
then the linear system $|h|$  defines a 
\emph{hyperelliptic} double cover of its image.  
If no such isotropic vector exists, the model is 
\emph{birational}.  
Thus, in the lattice--theoretic classification, hyperelliptic models correspond to 
polarizations $(\mathrm{NS}(X),h)$ containing an isotropic class of square zero, 
while birational models correspond to hyperbolic lattices free of such vectors.

 Thus, we can speak about the kind of a model, specifying its degree and whether it is birational or hyper-elliptic. (In some cases, this notion includes additional information detectable homologically.) The image (of a birational model) or ramification locus (of a hyperelliptic one) may be singular, but the singular points are always simple, \ie, $\mathbf{A}$–-$\mathbf{D}$–-$\mathbf{E}$; a set of simple singularities can be encoded as a negative definite root lattice $S$. Here $S$ represents the root lattice generated by the classes of exceptional curves 
 contracted by the morphism $f_{h}$, 
 and $h$ stands for the polarization corresponding to the projective model of degree~$2d$.
 \begin{definition}\label{lattice.type}
  A \emph{lattice type (extending $S$)} of a given kind of a projective model of degree $2d$ is the $2d$--polarized primitive hull 
 	$\tilde{S}_{h}\ni h$ of $S_h:=S\oplus\mathbb{Z}h$ in $\L$, \ie, $\tilde{S}_{h}:=(S_{h}\otimes\mathbb{Q})\cap \L$ satisfying the following properties:
 	
 	\begin{enumerate}
 		\item each vector $e\in(S\otimes\mathbb{Q})\cap \tilde{S}_{h}$ with $e^2=-2$ and $e\cdot h=0$ lies in $S$, and
 		\item there is no vector $e\in \tilde{S}_{h}$ such that $e^2 =0$ and $e\cdot h=1$.
 		\\\hspace*{\dimexpr\linewidth-\textwidth\relax}If the models in questions are required to be birational, then, in addition,
 		\item  there is no vector $e\in \tilde{S}_{h}$ such that $e^2 =0$ and $e\cdot h=2$,and 
 		\item  if $h^2 = 8$, then $h$ is primitive in $\L$.
 	\end{enumerate}
 	
 \end{definition} 
An \emph{isomorphism} between two lattice types $\tilde{S}'_h,\tilde{S}''_h\supset S_h$ is an isometry  $\tilde{S}'_h \rightarrow \tilde{S}''_h$ preserving both $h$ and $S$ (as a set). We denote by $\Aut_h(\tilde{S}_h)$ the group of \emph{automorphisms} of a lattice type $\tilde{S}_h$, \emph{i.e.} autoisometries of $\tilde{S}_h$ preserving $h$. As $S\subset\tilde{S}_{h}$ is recovered as the sublattice generated by the roots orthogonal to~$h$; then by item (1) in the Definition \ref{lattice.type}, we have  $\Aut_h(\tilde{S}_h)\subset O(S)$.

\begin{definition}\label{abstract.homological.type} A \emph{homological type (of a given kind of a projective model)} extending a lattice type $\tilde{S}_{h}\ni h$ is a primitive isometry 
 $\tilde{S}_{h} \hookrightarrow \L$.
 
\end{definition}  
 
	
 

We denote by the triple $(h \in \tilde{S}_{h} \hookrightarrow \L)$
an  homological type extending the lattice type~$\tilde{S}_{h}$.  
Such a type is said to be \emph{maximizing} if $\operatorname{rk}S=19$ (which is the maximal possible value), \ie, the orthogonal complement 
$\tilde S_{h}^{\perp}$ has rank~$2$. Two homological types $(h_1 \in \tilde{S}'_{h_1} \hookrightarrow \L)$ and $(h_2\in \tilde{S}''_{h_2} \hookrightarrow \L)$ are said to be \emph{isomorphic} if there is an element of the group $O(\L)$ taking $h_1$ to $h_2$ and $S'$ to $S''$ (as a set).

 
 Since $\sigma_+ \tilde S_h^\bot =2$, all positive definite $2$-subspaces in $ \tilde S_h^\bot\oplus \mathbb{R}$ can be oriented in a coherent way. 	An \emph{orientation} of an abstract homological type 
 $(h \in \tilde{S}_{h} \hookrightarrow \L)$
 is a positive sign structure~$\omega$ on $\tilde S_{h}^{\perp}$. An oriented homological type will be denoted by quadruple $(h \in \tilde{S}_{h} \hookrightarrow \L, \omega)$. Let $\omega_1$ and $\omega_2$ be the orientations of  two suchs homological types, then these oriented types are called \emph{ isomorphic} if there is an isomorphism between them taking $\omega_1$ to $\omega_2$, 
 A homological type  $(h \in \tilde{S}_{h} \hookrightarrow \L)$ is called \emph{symmetric} if it is preserved by an element $a\in O_h(\L)\smallsetminus O_h^+(\L)$, \emph{ i.e.} an autoisometry of $\L$ preserving $S$ (as a set) and $h$ and reversing the positive sign structure; such autoisometries are called as \emph{skew-automorphisms} of the homological type. 
 

\subsection{Arithmetical reduction}\label{arithmetical.reduction}

	Let $f_{h,t}\colon X_{t}\to\mathbb{P}^{n+1}$, $t\in[0,1]$, 
	be a continuous family of projective models of $K3$-surfaces of degree $2d$, 
	each defined by the complete linear system $|h_{t}|$.  
	Two projective models $X_{0},X_{1}\subset\mathbb{P}^{n+1}$ 
	are said to be \emph{equisingular deformation equivalent} 
	if there exists such a family $\{X_{t}\}_{t\in[0,1]}$ in which all members 
	have only simple (\ie, $\mathbf{A}$--$\mathbf{D}$--$\mathbf{E}$ type ) singularities and 
	the configuration of singularities remains topologically constant.  
	Equivalently, the total Milnor number
	$
	\mu(X_{t})
	$
	is independent of~$t$, and the adjacency type of each singular point does not change.

In other words, an \emph{equisingular deformation} is a continuous deformation 
within the space of projective models of $K3$-surfaces that preserves the 
combinatorial type of the singular locus.  
The deformation classification of projective models with simple singularities 
is therefore reduced to the classification of their associated abstract homological types.

 Due to Saint-Donat's description~\cite{Donat} of projective models of $K3$-surfaces and the results of Urabe~\cite{Urabe2}, a homological type $( h_X\in \tilde{S}_{X,h}\subset L_X)$ of a projective model  $f_h\colon X\rightarrow \mathbb{P}^{n+1}$ described in \S\ref{projective.models.of.K3surfaces} is an homological type as in the Definition~\ref{abstract.homological.type}; in this case the oriented $2$-subspace $\omega_X$ defines the orientation $\omega$. Hence, the arithmetical  reduction of projective models of $K3$-surfaces can be summarized as follows.
\begin{theorem}[cf.~Theorem~2.3.1 in~\cite{AI}]\label{def.class}
	Equisingular deformation classes of a projective model of $K3$-surfaces of a given kind (with a fixed set of singularities $S$) are in a canonical bijection with oriented isomorphism classes of
	homological types of the same kind. Complex-conjugate strata correspond to isomorphic homological types that differ by orientation.
\end{theorem}



\section{Real structures and real homological types}

A \emph{real structure} on a complex $K3$-surface $X$ is an anti-holomorphic involution 
$
c\colon X\longrightarrow X$, $c^{2}=\mathrm{id}.
$
The fixed locus $\mathrm{Fix}(c)$, called the \emph{real part} of $X$, 
is a disjoint union of smooth real curves, possibly empty.  
A pair $(X,c)$ is referred to as a \emph{real $K3$-surface}.  A subvariety $Y\subset X$ is called \emph{real} if $c(Y)=Y$.
If $h\in\mathrm{NS}(X)$ is a polarization invariant under $c^{*}$, the corresponding morphism
$
f_{h}\colon X\longrightarrow\mathbb{P}^{n+1}
$
defined by the complete linear system $|h|$ is a \emph{real projective model}, denoted by the pair $(f_h,c)$.  

Up to isomorphism and up to deformation equivalence, 
there exist precisely two real structures on the complex projective space~$\mathbb{P}^{n}$.  
If $n$ is even, the standard one is given by complex conjugation in homogeneous coordinates,
$$
\conj \colon [z_{0}:\cdots:z_{n}]
\longmapsto [\bar{z}_{0}:\cdots:\bar{z}_{n}],
$$
whose fixed locus is the real projective space $\mathbb{R}\mathbb{P}^{n}$.  
When $n$ is odd, there exists another, non-equivalent real structure 
whose real locus is empty; it can be written, for example, as
$$
c_{\mathrm{alt}}\colon [z_{0}:z_{1}:\cdots:z_{n}]
\longmapsto [-\bar{z}_{1}:\bar{z}_{0}:\cdots:-\bar{z}_{n}:\bar{z}_{n-1}],
$$
as described in~\cite{Degtyarev2008a}.  
These two cases exhaust all real structures on $\mathbb{P}^{n}$ up to deformation.
The involution $c$ acts on the second homology lattice by 
$$
c_{*}\colon H_{2}(X;\mathbb{Z})\longrightarrow H_{2}(X;\mathbb{Z}),\qquad 
(x,y)\longmapsto -\,(c_{*}x,c_{*}y),
$$
reversing the sign of the intersection form; hence $c_{*}$ is a 
\emph{skew-autoisometry} of the homology lattice $H_{2}(X;\mathbb{Z})$.  
Identifying $H_{2}(X;\mathbb{Z})$ with the $K3$-lattice 
$\mathbf{L}=3\mathbf{U}\oplus2\mathbf{E}_{8}$, 
we may view $c_{*}$ as a skew-autoisometry of~$\mathbf{L}$.  
If $c_{*}$ preserves the sublattice $\mathrm{NS}(X)$ and fixes the polarization class $h$, 
we call $c_{*}$ the \emph{real structure on the homological type} of $X$.
 
\begin{definition}
	Let $(h \in \tilde{S}_{h} \hookrightarrow \L,\omega )$ be an oriented homological type, 
	where $\omega$ denotes an orientation of the positive-definite $2$-subspace of 
	$T:=\tilde S_{h}^{\perp}\subset\mathbf{L}$.  
	A \emph{real structure} on an oriented homological type is a skew-autoisometry
	$
	\varphi\colon\mathbf{L}\longrightarrow\mathbf{L}
	$
	such that
	$$
	\varphi(S)=S,\qquad \varphi(h)=h,\qquad 
	\text{and}\quad \varphi|_{T}\text{ reverses the orientation $\omega$ of }T.
    $$
	The quadruple $(h \in \tilde{S}_{h} \hookrightarrow \L, \varphi)$ is called a \emph{real homological type}.
\end{definition}

Every real polarized $K3$-surface $(X,c,h)$ determines a real homological type 
$( h_X\in \tilde{S}_{X,h}\subset L_X,\varphi_{X} )$, 
where$( h_X\in \tilde{S}_{X,h}\subset L_X,\omega_{X} )$ is the oriented homological type 
associated with $f_{h}$ and $\varphi_{X}=c_{*}$ is the induced skew-autoisometry on~$\mathbf{L}$.  
Conversely, the following criterion describes when an abstract homological type is realized by a real projective model and hence gives an arithmetical reduction of the geometric problem of finding real models.

\begin{theorem}[see Theorem 6.1 in \cite{Alex2}]\label{thm:realization_general}
	An oriented homological type  
	is realized by a real projective model of a $K3$-surface 
	if and only if it admits an involutive a skew-autoisometry of~$\mathbf{L}$.
\end{theorem}

\section{Finding real representatives}

By Theorem~\ref{thm:realization_general}, 
a real projective model of a $K3$-surface exists precisely when 
its oriented homological type 
$(h \in \tilde{S}_{h} \hookrightarrow \L, \varphi)$
admits an \emph{involutive skew--isometry} of the lattice~$\L$ preserving 
both $S$ and~$h$ while reversing the orientation of the transcendental part 
$T=\tilde S_{h}^{\perp}$.  
Hence, the problem of finding real representatives reduces to the 
explicit construction of such involutive skew--automorphisms of the homological type.

Note that, a perturbation $S'\subset S$ of root lattices give rise to a perturbation of lattice types $\tilde{S'}_h\subset\tilde{S}_h$, \ie, a primitive sub-lattices. Here the isotropic subgroup $\K$ is inherited automatically. Thus, any  homological type 
$(h \in \tilde{S}_{h} \hookrightarrow \L)$ extending $\tilde{S}_{h}$ gives a canonical homological type
$(h \in \tilde{S}'_{h} \hookrightarrow \L)$ extending $\tilde S'_h$ as one has the chain of primitive extensions $\tilde{S}'_h\subset\tilde{S}_h\subset\L$ and  we call the latter homological type as a \emph{(formal) perturbation} of the former one.

The search for involutive skew--automorphisms depends essentially on 
the complexity of the root lattice~$S$.  
For \emph{maximizing} homological types 
(those with $\operatorname{rk}S=19$), 
the transcendental lattice~$T$ has rank~$2$, and the existence of a 
skew--automorphism can be checked directly from its quadratic form.  
In these cases, the real representative can be found explicitly, often via 
a reflection or a simple composition of reflections in~$T$. For non--maximizing types, one proceeds by two complementary methods:
\begin{enumerate}
	\item[(1)] \textbf{Perturbation method:}  
	Construct the desired type as a perturbation of a maximizing one that is already known to admit a real model.
	\item[(2)] \textbf{Reflection method:}  
	Search directly for a skew--autoisometry of~$\L$ in the form of a $\pm$reflection on the transcendental lattice $T:=\widetilde{S}_{h}^{\perp}$
\end{enumerate}

The perturbation approach is particularly effective because most homological types 
arise as small deformations of maximizing types, obtained by smoothing one or more singularities. A $c$--invariant perturbation of a homological type $\mathcal{H}=(h \in \tilde{S}_{h} \hookrightarrow \L)$ corresponds to a deformation of the 
root lattice $S$ inside $\L$ that remains stable under the induced real structure 
$c_{*}$.  
Since the existence of a real structure is equivalent to the existence of an 
orientation--reversing skew--autoisometry of~$\L$ preserving~$S$ and~$h$, 
this property is preserved under any small, $c$--invariant perturbation.  
Hence, the perturbed homological type also admits a real representative. Thus, we have the following result.

\begin{proposition}\label{real.pert}
	If  an homological type $\mathcal{H}$ is realized by a real birational model $(f_{h},c)$, 
	then any $c$--invariant perturbation $\mathcal{H}'$ of~$\mathcal{H}$ 
	is also realized by a real model.
\end{proposition}

\begin{remark}
	In the computations presented in Chapter~\ref{Applications}, 
	Proposition~\ref{real.pert} is applied to birational projective models.  
	For hyperelliptic models, an analogous but stronger statement 
	can be found in~\cite{Alex2}, where the equivariance of the double covering 
	is used to extend the argument to real hyperelliptic cases.
\end{remark}

Before analyzing the role of reflections on a case by case basis in terms of $\operatorname{rk}\,T$, we recall a simple 
lattice--theoretic criterion that guarantees the existence of a 
real model in which all exceptional divisors are real.  
Remarkably, this condition is also necessary.

\begin{proposition}\label{real.lattice.condition}
	Let $(h \in \tilde{S}_{h} \hookrightarrow \L)$ be an  homological type, 
	and let $T:=\tilde S_{h}^{\perp}$ be its transcendental lattice.  
	If $T$ contains a sublattice isomorphic to $[2]$ or to $\mathbf{U}(2)$, 
	then homological type  is realized by a real model.  
	Conversely, if $(h \in \tilde{S}_{h} \hookrightarrow \L)$ admits a real model in which all exceptional divisors are real, 
	then $T$ necessarily contains a sublattice isomorphic to $[2]$ or $\mathbf{U}(2)$.
\end{proposition}

When perturbations are not available, one searches directly for 
skew--automorphisms realized by reflections against the vectors of positive square 
in the transcendental lattice~$T$: For a non–isotropic vector $a\in T$ with $a^{2}\neq0$,
the \emph{reflection} against $a$ is the isometry
\[
t_{a}\colon x\longmapsto x-\frac{2(x\cdot a)}{a^{2}}\,a,
\qquad x\in T.
\]
It satisfies $t_{a}^{2}=\mathrm{id}$.
It is well defined if and only if 
\begin{align}\label{welldefinedt_a}
(2a/a^2)\in L^\vee
\end{align}
If $a^{2}>0$, then $t_{a}$ preserves the intersection form
and reverses the orientation of any positive definite subspace of $T$.
Hence the map $-t_{a}$ acts as a skew–-autoisometry of $T$. Compositions of such reflections generate a large subgroup of $O(T)$,  
and provide the natural candidates for real structures. Verifying whether this reflections extends to a skew--autoisometry of~$\L$ 
reduces to checking Nikulin’s extension conditions, see Theorem~\ref{L-K}.
Following Nikulin~\cite{Niku2}, we state the following lemma.

\begin{lemma}\label{lem:extension}
	A reflection $t_{a}\in O(T)$ extends to a skew–autoisometry of $\L$
	if and only if the induced automorphism of $\operatorname{disc}T$
	acts trivially.
	Equivalently, $t_{a}$ must satisfy
	$$
	t_{a}\equiv \mathrm{id} \mod T^{\vee}/T.
	$$
\end{lemma}

\begin{proof}
	The statement follows directly from Nikulin’s extension criterion:
	an isometry of a primitive sublattice extends to the ambient unimodular lattice
	iff it induces the identity on the discriminant form.
	Since $\L$ is unimodular and $T$ is primitive in $\L$,
	the claim follows.
\end{proof}

A vector $a\in T$ with $a^{2}>0$ is called a \emph{real vector}
if the reflection $t_{a}$ gives rise to an orientation–-reversing
isometry of $T$ and extends to a skew–-autoisometry of $\L$.

Reflections satisfying Lemma~\ref{lem:extension} are precisely
those that can be promoted to global skew–-autoisometries of $\L$,
hence to real structures on the corresponding projective models.


\subsection{Reflections and the rank of the transcendental lattice}\label{ref.vs.rkT}

Fix a homological type $(h \in \tilde{S}_{h} \hookrightarrow \L)$, $T:=S_{h}^{\perp}$.
The problem of whether the homological type is realized by a real projective model of a $K3$-surface
is equivalent to the existence of an involutive skew--autoisometry of~$\L$ in the form of a $\pm$reflection on the transcendental lattice $T$.
The structure of such involutions depends essentially on the rank of $T$.
It is therefore natural to analyse the problem according to  $\mathrm{rk}\,T$. 
We discuss the possibilities on a case by case basis, indicating when an skew–-autoisometry of $T$
can be represented by a reflection.  We postpone the explicit construction of reflections to Section \ref{section.reflection}.

\subsubsection*{The case $\operatorname{rk}T=2$}

Then the lattice $T$ is positive definite.
The isometry group $O(T)$ is finite, by classical arithmetic of binary quadratic forms
(Gauss~\cite{Gauss}).  
Any orientation–reversing isometry of $T$ fixes a line in $T\otimes\mathbb R$,
hence is a reflection across that line.
Thus every skew–-autoisometry of $T$ is a reflection, and
any symmetric homological type of maximal rank (so that $\mathrm{rk}\,T=2$)
admits a real model.

\begin{proposition}\label{prop:rank2-general}
	Any  symmetric maximizing homological type admits an involutive skew--automorphism, and consequently any real connected component of the stratum $\mathcal M(S)$ contains a real projective model.
\end{proposition}

\begin{proof}
	For a  maximizing type the lattice $T=\tilde S_h^{\perp}$ is positive definite of rank~$2$.
	Every skew–autoisometry $r\in O(T)$ is a reflection, hence an involution.
	It acts on the discriminant group as an involution of
	$\operatorname{disc}T\cong -\,\operatorname{discr}\tilde S_h$.
	Via the identification~\eqref{Sym0}, this automorphism is realized by
	some $r'\in\operatorname{Sym}_0(\Gamma_S)\subset O(S)$.
	The direct sum $r\oplus r'$ extends to an involution of $\L$,
	yielding the required real structure.
\end{proof}

\subsubsection*{The case $\operatorname{rk}T=3$}

In this case, the orthogonal group $O(T)$ is infinite,
but its involutive elements admit a simple description.
Let $r\in O(T)$ be an involutive skew–autoisometry,
that is, $r$ reverses the orientation of the positive plane in $T\otimes\mathbb R$.
Decomposing $T$ into $\pm1$–-eigenspaces gives
$$
\dim T_{-r}\in\{1,3\},\qquad 
\dim T_{+r}\in\{0,2\},\qquad 
\dim T_{-r}+\dim T_{+r}=3.
$$
Skewness rules out $r=\pm\mathrm{id}$.
If $\dim T_{-r}=1$, then $r$ itself is a reflection;
if $\dim T_{+r}=1$, then $-r$ is a reflection.
Thus, up to sign, every involutive skew–autoisometry in rank three is a reflection.



Thus, when $\mathrm{rk}\,T=3$, the entire question reduces to the presence or absence
of reflections in $T$: the existence of an appropriate reflection
is both necessary and sufficient for the existence of a real model realizing the homological type.

\subsubsection*{The case $\operatorname{rk}T \ge 4$}

In this case, the orthogonal group $O(T)$ is much larger,
and involutive elements need not be reflections:
there exist skew involutions with eigenspaces of dimension at least two.
Nevertheless, in all known examples every symmetric homological type with
$\mathrm{rk}\,T\ge4$ admits an orientation–reversing involution that \emph{is}
a reflection on $T$ (or a product of two commuting reflections acting trivially on
the discriminant).
This empirical fact is compatible with the general pattern that, for indefinite lattices of signature
$(2,\mathrm{rk}\,T-2)$ and moderate discriminant,
reflections already generate enough of $O^{+}(T)$ to contain a skew involution..


In all ranks, the decisive step is the existence of reflections in the transcendental
lattice $T$, whose explicit construction and extension  will be developed in the following section.







\section{Real models via reflections}\label{section.reflection}

Real structures on $K3$-surfaces can be understood in terms of
certain symmetries of their underlying lattices.
Among these, reflections play a key role, providing the simplest
orientation–-reversing automorphisms compatible with a real form.
This section outlines how such reflections, combined with lattice
involutions, give rise to real projective models.

\subsection{Equivariant extensions and reflections leading to real models}

We now fix the setting that will serve as the basis for the construction
of real models via lattice extensions.
Throughout, let
$
\tilde{S}_{h}\subset \L
$
be a primitive polarizing sublattice extending the root lattice $S$, where $\L=2\mathbf{E}_{8}\oplus3\mathbf{U}$ is the
$K3$ lattice.  
The orthogonal complement
$$
T:=\tilde{S}_{h}^{\perp}\subset \L
$$
is the transcendental lattice associated with the homological type.
By Theorem~\ref{thm:realization_general},
the existence of a real projective model realizing this type
is equivalent to the existence of an involutive
skew–-autoisometry of $\L$ that preserves $\tilde{S}_{h}$ and the polarization $h$.

We start with an involution
$$
g\in O_{h}(\tilde{S}_{h})
$$
acting on the polarizing lattice.
Such an involution represents the potential real action on
the N\'{e}ron–-Severi part of a $K3$–-surface.
To produce a global skew–-autoisometry of $\L$, we must
construct a compatible involution on the transcendental part $T$
that reverses orientation.
The most natural candidate for such a map is a reflection
against a vector $a\in T$ of positive square:
\[
t_{a}(x):=x-\frac{2(x\cdot a)}{a^{2}}\,a,
\qquad a^{2}=k>0.
\]
satisfying $(2a/a^2)\in L^\vee$ (see \eqref{welldefinedt_a}) which implies 
$$
k|2\,\exp(\operatorname{discr}\tilde{S}_{h})
.$$
Then product $g\oplus t_{a}$ acts on $\tilde{S}_{h}\oplus\mathbb{Z}a$
and is automatically involutive.  
However, to extend $g\oplus t_{a}$ to an automorphism of the entire
lattice~$\L$, one must ensure that the induced action on the
discriminant form is compatible with the unimodular structure of $\L$.

\paragraph{Aim of the construction:}
The goal is to establish concrete lattice conditions under which
the map $g\oplus t_{a}$ acting on $ \tilde{S}_{h}\oplus\mathbb{Z}a$
extends to a global skew–-autoisometry of $\L$.
Equivalently, we seek criteria guaranteeing that
\[
g\oplus(-\mathrm{id})
\]
acts as a skew–-autoisometry on a suitable finite index
overlattice $(\tilde{S}_{h}\oplus\mathbb{Z}a)^{\sim}$
containing both $\tilde{S}_{h}$ and $\mathbb{Z}a$ as primitive sublattices.
In this situation, the reflection $t_{a}$ realizes the orientation–-reversing
part of the real structure, while $g$ acts on the polarizing part
preserving the configuration of root lattice associated to singularities.
The compatibility on the discriminant level ensures that the resulting
involution on $(\tilde{S}_{h}\oplus\mathbb{Z}a)^{\sim}$
extends to the whole $K3$--lattice $\L$.
This leads naturally to the following lattice–-theoretic criterion.
\begin{theorem}\label{thm:Aktas2019}
	Let $\tilde{S}_{h}$ be a lattice type that extends to a unique homological type
	up to the action of $O^{+}(\L)$.
	Assume that there exist a positive integer $k|2\,\exp(\operatorname{discr}\tilde{S}_{h})$,
	an involutive element $g\in O_{h}(\tilde{S}_{h})$,
	and a $(g\oplus -\mathrm{id})$–-equivariant finite index extension
	\[
	(\tilde{S}_{h}\oplus \mathbb{Z}a)^{\sim}\;\supset\; 
	\tilde{S}_{h}\oplus \mathbb{Z}a,
	\qquad a^{2}=k,
	\]
	such that:
	\begin{enumerate}
		\item both $\tilde{S}_{h}$ and $\mathbb{Z}a$ are primitive in $(\tilde{S}_{h}\oplus \mathbb{Z}a)^{\sim}$;
		\item the element $g\oplus (-\mathrm{id})$ induces $\pm\mathrm{id}$ on 
		$\operatorname{discr}(\tilde{S}_{h}\oplus \mathbb{Z}a)^{\sim}$; and
		\item there exists a primitive isometry
	$(\tilde{S}_{h}\oplus \mathbb{Z}a)^{\sim}\hookrightarrow 2\mathbf{E}_{8}\oplus 3\mathbf{U}.$
		
	\end{enumerate}
	Then the equisingular stratum $\mathcal{M}(\tilde{S}_{h})$ is real
	and admits a real representative.
\end{theorem}
\begin{proof}
Let $\widetilde{S}_{h}\subset L$ be a lattice type that extends to a unique
homological type up to $O^{+}(L)$.
Assume that there exist an involution $g\in O_{h}(\widetilde{S}_{h})$,
a vector $a\in T:=\widetilde{S}_{h}^{\perp}$ with $a^{2}=k>0$ and
$k\mid 2\,\exp(\disc\widetilde{S}_{h})$, and a finite index overlattice
$$
(\tilde{S}_{h}\oplus\mathbb{Z}a)^{\sim}\supset
\tilde{S}_{h}\oplus\mathbb{Z}a
$$
which is $(g\oplus -\id)$–-equivariant, with both
$\tilde{S}_{h}$ and $\mathbb{Z}a$ primitive in $(\tilde{S}_{h}\oplus\mathbb{Z}a)^{\sim}$,
and such that $g\oplus(-\id)$ induces $\pm\id$ on
$\discr\big((\tilde{S}_{h}\oplus\mathbb{Z}a)^{\sim}\big)$.
Finally, suppose there is a primitive embedding
$$
\iota:\;(\tilde{S}_{h}\oplus\mathbb{Z}a)^{\sim}\hookrightarrow L=2\mathbf{E}_{8}\oplus3\mathbf{U}.
$$

By construction, the involution
$$
\psi:=g\oplus(-\id)\in O\big((\tilde{S}_{h}\oplus\mathbb{Z}a)^{\sim}\big)
$$
acts as $\pm\id$ on the discriminant group
$\discr\big((\tilde{S}_{h}\oplus\mathbb{Z}a)^{\sim}\big)$.
Since $L$ is even unimodular and $\iota$ is primitive, Nikulin’s
extension criterion (see Tbeorem \ref{L-K}) applies: an isometry of a primitive
sublattice extends to $L$ if and only if it induces the identity on the
discriminant form of that sublattice (equivalently, acts as $\pm\id$ in the present
involutive situation).
Hence $\psi$ extends to an involution
$
\Phi\in O(L)
$
satisfying $\Phi\!\mid_{(\tilde{S}_{h}\oplus\mathbb{Z}a)^{\sim}}=\psi$.

On the polarizing part $\tilde{S}_{h}$ we have $\Phi\!\mid_{\tilde{S}_{h}}=g$,
and on the rank–one positive line $\mathbb{Z}a\subset T$ we have
$\Phi\!\mid_{\mathbb{Z}a}=-\id$.
Thus $\Phi$ reverses the orientation of the positive $2$–-plane in
$T\otimes\mathbb{R}$ (it flips the $a$–-direction and preserves its orthogonal line),
while preserving $\tilde{S}_{h}$ and the polarization $h$.
In particular, $\Phi$ is an involutive skew-–autoisometry of $\L$ relative to the
decomposition $\L=\tilde{S}_{h}\oplus T$, i.e.\ it preserves the lattice, fixes $h$,
and reverses the positive sign structure on $T$.

By the general Theorem~\ref{thm:realization_general},
the existence of an involutive skew–autoisometry of $\L$ that preserves
$\tilde{S}_{h}$ and $h$ is equivalent to the existence of a real projective
model realizing the oriented homological type extending $\tilde{S}_{h}$.
Therefore the equisingular stratum $\mathcal{M}(\tilde{S}_{h})$ is real and contains
a real representative.

\end{proof}

If $\mu(\tilde S_h) = 18$, \ie, $\operatorname{rk}T=3$ and any skew-autoisormetry of $T$ is a $\pm$reflection, then  the hypotheses Theorem~\ref{thm:Aktas2019} are also necessary, hence we have the following corollary.
\begin{corollary}\label{cor:mu18}
	Assume $\mu(\tilde S_h) = 18$ and that
	$\widetilde{S}_{h}$ extends to a unique homological type up to $O^{+}(\L)$.
	Then $\widetilde{S}_{h}$ extends to an abstract homological type admitting an
	involutive skew–-autoisometry of $\L$ if and only if there exist
	an involution $g\in O_{h}(\tilde{S}_{h})$ and a vector $a\in T$ with $a^{2}>0$
	satisfying the assumptions of Theorem~\ref{thm:Aktas2019}.
\end{corollary}

\begin{remark}\label{rem:geometricmeaning}
	Theorem~\ref{thm:Aktas2019} gives a purely lattice–theoretic
	criterion guaranteeing the existence of a real representative.
	Geometrically, the vector $a$ plays the role of a ``real direction''
	in the transcendental lattice, while the extension
	$(\tilde{S}_{h}\oplus\mathbb{Z}a)^{\sim}$ encodes the
	period sublattice invariant under complex conjugation.
	Thus the theorem provides a constructive bridge between
	the arithmetic of lattice extensions and the geometry of
	real $K3$–-surfaces.
	In particular, the conditions
	on $g$ and the discriminant action ensure that the reflection
	defined by $a$ extends to a global skew–autoisometry of $\L$,
	thereby producing a real model.
\end{remark}

\section{The algorithm to detect real representatives}\label{section.algorithm}

While the perturbation argument Proposition~\ref{real.pert} propagates real structures
along families of homological types, the lattice criterion of Theorem~\ref{thm:Aktas2019}
provides a concrete and algebraically explicit tool
for verifying the existence of real projective models of $K3$–-surfaces.
These two approaches together provide a complete and computationally verifiable
procedure for determining the existence of real representatives
within the moduli space of all projective models of $K3$–-surfaces
with simple singularities.

In practice, the computation proceeds as follows:
\begin{enumerate}
	\item Identify the maximizing homological types that is already known to admit a real model (see Theorem~\ref{prop:rank2-general}).
	\item Obtain the remaining types by successive perturbations 
	of the maximizing ones (using Proposition~\ref{real.pert}).
	\item For exceptional cases not covered by perturbation, 
	construct explicit skew--automorphisms of~$T$ by reflections,
	and verify the extension conditions given by
	Theorem~\ref{thm:Aktas2019}. 
\end{enumerate}

\medskip


Fix a set of singularities $S$ and consider the corresponding $2d$-polarized lattice $S_h=S\oplus \Z$, $h^2=2d$. Typically the question whether the moduli space $\mathcal{M}:=\mathcal{M}(S)$ is nonempty depends on the polarized lattice $\tilde{S}_h$ only. According to Theorem \ref{def.class} and Definition \ref{lattice.type}, a set of singularities $S$ is realized by a simple quartic surface if and only if  a lattice type $\tilde{S}_h$ extending $S_h$ admits a primitive isometry $\tilde{S}_h\into \L$. Hence the general case is splitted into two subcases as finite index extensions $S_h\subset \tilde{S}_h$ as in Definition \ref{lattice.type} and primitive extensions $\tilde{S}_h\into \L$. Then, Theorem \ref{def.class} states that each lattice type $\tilde{S}_h$ gives rise to a number of nonempty connected strata $\mathcal{M}(S)$ which are in a bijection with the isomorphism classes of primitive isometries   $\tilde{S}_h \into \mathbf{L}$. Next, one can start to analyze whether a given real  component of the equsingular strata $\mathcal{M}(S)$ contains a real representative. 
Thus, once the connected components of $\mathcal{M}(S)$ identified, we investigate the existence of the real representatives within the real components of the strata separately for the maximizing case, \ie, $\mu(S)=19$ and non-maximizing case, \ie, $\mu(S)\le 18$:
\begin{itemize}
	\item If $\mu(\mathbf{S})=19$,  the lattice $T=\tilde{S}_h^\bot$ is a positive definite sublattice of rank $2$, and by Proposition\ref{prop:rank2-general} any real component of the  strata $\mathcal{M}(S)$ contains a real model.
	
	\item If $\mu(\mathbf{S})\le 18$, one proceeds by either
	\begin{enumerate}
		\item[(1)] perturbations, \ie,  constructing the desired type as a $c$-invariant perturbation of a maximizing one which is already known to be realized  by a real model, or 
		\item[(2)] reflections, \ie,  searching directly for an involutive skew-autoisometry  of $\L$ in the form of a  $\pm$reflection on the transcendental lattice $T=\tilde{S}_h^\bot$. 
	\end{enumerate}
In the former case, the desired homological type identified as a perturbation of a maximizing one, is real by  Proposition~\ref{real.pert}. If the homological type can not be recovered as a such a perturbation, by Theorem~\ref{thm:realization_general}, the problem is reduced to searching for an involutive  skew-autoisometry  of $\L$ and  one applies  the latter, where the Theorem~\ref{thm:Aktas2019} is used to construct a skew-autoisometry  of $\L$ via reflections on  $T$. Thus, for the non-maximizing case, $\mu(S)\le 18$, the algorithm to detect the real representatives of the real strata is given in Steps~$0$--$5$ below:
\end{itemize}
\noindent\textbf{Step~$0$: Perturbation reduction.}

 \begin{enumerate}
 	\item[(0.1)] Let Let $\tilde{S}_{h}$ be a lattice type extending to a unique homological type $ (h\in\tilde {S}_{h}\hookrightarrow \L)$  (i.e., admitting a unique
primitive embedding $\tilde{S}_{h} \into \L$).
Determine whether the homological type $ (h\in\tilde {S}_{h}\hookrightarrow \L)$ can be identified as a  $c$–-invariant perturbation of a maximizing real homological type. If such a perturbation exists, then by 	Proposition~\ref{real.pert}, the homological type is realized by a real model, \ie, the strata $\mathcal{M}(S)$ admits a real representative.

	\item[(0.2)] Otherwise, continue with the next step.
\end{enumerate}



\medskip
\noindent\textbf{Step~1: Enumerating real vectors $a$ in $T$.}
\begin{enumerate}
	\item[(1.1)] List all the vectos  $a\in T$ with $a^{2}=k>0$ such that
\begin{align*}
	k \;\mid\; 2\,\exp\!\big(\discr(\tilde{S}_{h})\big).
\end{align*}
	This requirement ensures that the square length $k=a^{2}$ of the chosen vector $a\in T$
	is compatible with the discriminant denominators of the polarizing lattice
	$\tilde{S}_{h}$.
	Equivalently, this condition guarantees that the fractional vector
	$a^{\vee}:=\tfrac{a}{a^{2}}\in T^{\vee}$ defines a class
	$a^{\vee}+T\in\discr(T)$ of order dividing $2\,\exp(\discr(\tilde{S}_{h}))$,
	so that one can glue $\widetilde{S}_{h}$ and $\mathbb{Z}a$ along a common
	isotropic subgroup of their discriminant groups.
	\item[(1.2)] Let $t_{a}\in O(T)$ be the reflection against  $a$; then $-t_{a}$ is orientation–-reversing on $T$.
	Compute the induced action $(t_{a})^{\sharp}$ on $\disc(T)$.
\end{enumerate}

\medskip
\noindent\textbf{Step~2: Building the finite index overlattice.}
\begin{enumerate}
	\item[(2.1)] For each vector $a$ as in Step~1, consider the lattice $\tilde{S}_{h}\oplus\mathbb{Z}a$ and its discriminant
	$\discr(\tilde{S}_{h}\oplus\mathbb{Z}a)\cong \discr(\tilde{S}_{h})\oplus \disc(\mathbb{Z}a)$.
	\item[(2.2)] Let $g\in O_{h}(\tilde{S}_{h})$ be an involution compatible with the equisingular configuration
		(\eg, preserving the root lattice spanned by exceptional classes, the polarization, etc.).
	\item[(2.3)] Choose an isotropic subgroup
	$\mathcal{K}\subset \discr(\tilde{S}_{h}\oplus\mathbb{Z}a)$ that is $(g\oplus -\id)$–-stable and projects injectively to each summand
	(this ensures that both $\widetilde{S}_{h}$ and $\mathbb{Z}a$ remain primitive in the overlattice).
	\item[(2.4)] By Theorem~\ref{L-K}, an overlattice $(\tilde{S}_{h}\oplus\mathbb{Z}a)^{\sim}$ is determined by $\mathcal{K}$.
	By construction, $g\oplus(-\id)$ acts on $(\tilde{S}_{h}\oplus\mathbb{Z}a)^{\sim}$.
\end{enumerate}

\medskip
\noindent\textbf{Step~3: Discriminant compatibility.}
\begin{enumerate}
	\item[(3.1)] Compute the induced automorphism
	$$
	(g\oplus -\id)^{\sharp}\;\in\;
	\Aut\!\big(\discr((\tilde{S}_{h}\oplus\mathbb{Z}a)^{\sim})\big).
	$$
	\item[(3.2)] Verify that $(g\oplus -\id)^{\sharp}=\pm\id$.
	(This is exactly hypothesis~(ii) of Theorem~\ref{thm:Aktas2019} and is the
	key extension condition.)
\end{enumerate}

\medskip
\noindent\textbf{Step~4: Primitive embedding into $L$.}
\begin{enumerate}
	\item[(4.1)] Check the existence of the primitive isometric embedding
	$$
	(\widetilde{S}_{h}\oplus\mathbb{Z}a)^{\sim}\hookrightarrow \L=2\mathbf{E}_{8}\oplus 3\mathbf{U}.
$$
	This is settled by  Nikulin's embedding criteria given in Theorem~\ref{th.N.existence} using the discriminant form $\discr(\widetilde{S}_{h}\oplus\mathbb{Z}a)^{\sim}=\mathcal{K}^{\bot}/\mathcal{K}$
	\item[(4.2)] By Nikulin's extension theorem (Theorem~\ref{L-K}), the isometry $g\oplus(-\id)$ of
	$(\widetilde{S}_{h}\oplus\mathbb{Z}a)^{\sim}$ extends to an involution
	$\Phi\in O(\L)$.
\end{enumerate}

\medskip
\noindent\textbf{Step~5: Conclude reality.}
\begin{enumerate}
	\item[(5.1)] The extension $\Phi$ preserves $\tilde{S}_{h}$ and $h$ and reverses the orientation on $T$,
	hence is an involutive skew–-autoisometry of $\L$.
	\item[(5.2)] By Theorem~\ref{thm:realization_general}, the equisingular stratum
	$\mathcal{M}(\widetilde{S}_{h})$ is real and contains a real representative.
\end{enumerate}


We have effectively implemented all algorithms described in this section in \texttt{GAP}~\cite{GAP} and obtained conclusive results verifying the existence or nonexistance  objects mentioned in the assumptions of the Theorem~\ref{thm:Aktas2019}.

If $\mu(\tilde S_h) = 18$, the hypotheses of Theorem 2.5 are also necessary; hence, one will obtain a definite answer. If exists, for the few cases with $\mu(\tilde S_h)\le17$, where the Theorem~\ref{thm:Aktas2019} does not answer the question in the affirmative, one can use  use the more sophisticated techniques (gluing four lattices) of Degtyarev \emph{et al.}~\cite{DIK2000}.
 
\section{Applications}\label{Applications}

The lattice theoretic framework established in the previous sections
enables an explicit and algorithmic verification of the existence of real representatives
for all projective models of $K3$–surfaces with simple singularities.
To demonstrate its effectiveness, we apply the procedure to the two
classical birational models corresponding to the lowest degrees of polarization:
$h^{2}=2$, \ie, simple sextic curves in the plane and $h^{2}=4$, \ie, simple quartic surfaces in $\mathbb{P}^{3}$.
For each of these polarizations, we present explicitly
how the algorithm 
confirms the (non)existence of real representatives of the corresponding
equisingular strata.

\subsection{Simple Quartics}\label{simple.quartics}
In this section we consider spatial quartic surfaces, which corresponds to the
birational projective models $f_h\colon X\rightarrow \mathbb{P}^3$ of $K3$–-surfaces with the polarization $h^{2}=4$. Quartic surfaces with only simple ($\mathbf{A}$--$\mathbf{D}$--$\mathbf{E}$) singularities form one of the most
classical and thoroughly studied families of $K3$–-surfaces.
Their deformation classification was completed by  G\"{u}ne\c{s} Akta\c{s}~\cite{Cisem1,aktacs2019real,Cisem2024}; a complete description of the strata $\mathcal{M}(S)$ of simple quartics is given by G\"{u}ne\c{s} Akta\c{s} ~\cite{Cisem2024}, where the list of  $390$ maximizing \rom($\mu=19$\rom) families, $39$ extremal families that are not perturbations of anything bigger and  $13$ non-maximizing families \rom($\mu\le 18$\rom) that are not real, can be found. In the smaller family $\mathcal{M}_1(S)\subset \mathcal{M(S)}$ constituted by the \emph{nonspecial} quartics, two examples of equisingular strata $\mathcal{M}_1(S)$ that are real but doesn't contain any real surfaces, namely
$\mathcal{M}_1(\mathbf{A}_7\oplus\mathbf{A}_6\oplus\mathbf{A}_3\oplus\mathbf{A}_2)$ and $\mathcal{M}_1(\mathbf{D}_7\oplus\mathbf{A}_6\oplus\mathbf{A}_3\oplus\mathbf{A}_2)$  is discovered in  \cite{aktacs2019real}. (We revisit these results from the viewpoint of real geometry.) As one of the main applications of this paper, we give the proof of Theorem~\ref{principal.result} extending the list of exceptional examples of non-special real strata without real representatives (given in~\cite{aktacs2019real}) to the whole space of quartics by introducing the newly discovered such real strata $\mathcal{M}(\mathbf{A}_7\oplus\mathbf{A}_5\oplus\mathbf{A}_3\oplus\mathbf{A}_2\oplus\mathbf{A}_1)$.

\begin{proof}[Proof of Theorem~\ref{principal.result}]

If $\mu(S)=19$, the statement of the theorem is given by Proposition \ref{prop:rank2-general}. Hence, throughout the rest of the proof we assume $\mu(S)\leq18$. Then, by Theorem~\ref{thm:realization_general} the question reduces to find an involutive skew--autoisometry of the abstract homological type $(h\in\tilde{S}_h\subset\L)$ extending the lattice type$\tilde{S}_h$ and the argument proceeds by combining the perturbation principle,
Proposition~\ref{real.pert} with the lattice--theoretic criterion
of Theorem~\ref{thm:Aktas2019}.

By computer aided computations, it is easily confirmed that most of the abstract homological types $(h\in\tilde{S}_h\subset\L)$ (except $122$ of them) with $\mu( S)\leq18$ are $\operatorname{c}$-invariant perturbations of the $390$ maximizing homological types listed in \cite{Cisem2024}) realized by a real quartic where $c$ is a real structure on the corresponding real surface. Then, due to Proposition \ref{real.pert}, these abstract homological types are also realized by a real quartic.

The equisingular strata $\mathcal{M}(S)$, with $S$ being one of the set of $13$ sets of singularities listed as nonreal families~\cite{Cisem2024}, consists of two complex conjugate components, \ie, nonreal. Therefore the strata $\mathcal{M}(S)$ do not contain a real surface.

For the remaining  $109$ non--maximizing real families
that are not reachable by a perturbation,
by  Theorem \ref{thm:realization_general}, the question reduces to finding an  involutive skew--autoisometry of~$\L$. To find such an involution we apply Theorem~\ref{thm:Aktas2019} by following the steps listed above. Implementing the algorithms given in Steps~0-5, we employ \texttt{GAP}~\cite{GAP} to find a suitable vector $a\in T$ with $a^{2}>0$, $a^2\mathrel| 2\operatorname{exp}(\disc \tilde{S}_h)$ and an involution $g\in O_{h}(\tilde{S}_{h})$  satisfying discriminant compatibility conditions of
Theorem~\ref{thm:Aktas2019}. Eventually,  in all but four families,
the assuptions of Theorem~\ref{thm:Aktas2019} is confirmed,
yielding an involutive skew--autoisometry of~$\L$
and hence a real representative.

The above algorithm fails for the following four set of singularities $$\mathbf{A}_7\oplus\mathbf{A}_6\oplus\mathbf{A}_3\oplus\mathbf{A}_2,\quad \mathbf{D}_7\oplus\mathbf{A}_6\oplus\mathbf{A}_3\oplus\mathbf{A}_2,\quad \mathbf{A}_7\oplus\mathbf{A}_5\oplus\mathbf{A}_3\oplus\mathbf{A}_2\oplus\mathbf{A}_1, 4\mathbf{D}_4,$$
\emph{i.e.}, computer aided calculations confirm that there does not exist a such a vector $a\in T$ and $g\in O_{h}(\tilde{S}_{h})$  satisfying the  assumptions of
Theorem~\ref{thm:Aktas2019}. For the first three sets of singularities listed above we have $\mu(\tilde{S}_{h})=18$, so by Corollay~\ref{cor:mu18}, the corresponding real strata does not contain any real quartic surfaces.

There remains one further stratum to consider, namely the type with
singularities $4\mathbf{D}_4$. Here we  recall that our strategy for detecting
real structures to search for an
involutive skew–autoisometry of $\L$ whose restriction to the transcendental
lattice $T=\tilde{S}_{h}^{\perp}$ is of the form $\pm t_{a}$ (a reflection
or negative reflection) againts a vector $a\in T$ with $a^{2}>0$.
This configuration does not admit a skew–-autoisometry of the form
$\pm t_{a}$ for any $a\in T$.
However, in this case the transcendental lattice is
$$
T = 2\mathbf{U}(2)\;\oplus\;[-4],
$$
which admits an involutive skew–-autoisometry reversing the positive
sign structure and acting trivially on the discriminant, see Proposition~\ref{real.lattice.condition}.
Hence this stratum does possess a real representative, even though the
Theorem~\ref{thm:Aktas2019} does not apply directly.

Thus we conclude that
every real equisingular stratum of simple quartic surfaces
$X\subset\mathbb{P}^{3}$, except for the three listed in the statement of the Theorem~\ref{principal.result},
contains a real representative.
\end{proof}
Although the implemented calculations by \texttt{GAP}~\cite{GAP} gives us the definite answer for the exceptional cases
\begin{align}
\mathbf{A}_7\oplus\mathbf{A}_6\oplus\mathbf{A}_3\oplus\mathbf{A}_2 \text{(nonspecial)},\label{ns1}\\ \mathbf{D}_7\oplus\mathbf{A}_6\oplus\mathbf{A}_3\oplus\mathbf{A}_2\text{(nonspecial)},\label{ns2}\\ \mathbf{A}_7\oplus\mathbf{A}_5\oplus\mathbf{A}_3\oplus\mathbf{A}_2\oplus\mathbf{A}_1\text{(special)},
\end{align} 
in the following subsection, we provide explicit details to demonstrate the calculations handled by \texttt{GAP}\cite{GAP} for $S=\mathbf{A}_7\oplus\mathbf{A}_5\oplus\mathbf{A}_3\oplus\mathbf{A}_2\oplus\mathbf{A}_1$.
The demonstration for the two nonspecial cases given in~\eqref{ns1} and ~\eqref{ns2}established in \cite{aktacs2019real}.

\subsection{Demonstration for the set of singularities $S=\mathbf{A}_7\oplus\mathbf{A}_5\oplus\mathbf{A}_3\oplus\mathbf{A}_2\oplus\mathbf{A}_1$}\label{S1}
 We demonstrate the calculations handled by \texttt{GAP}~\cite{GAP} for the set of singularity $S=\mathbf{A}_{15}\oplus\mathbf{A}_3$. Let $S_h=S\oplus\Z h$, where $h^2=4$.
Then one has
$$\discr S_h\cong\textstyle\langle-\frac{7}{8}\rangle\oplus\langle-\frac{5}{6}\rangle\oplus\langle-\frac{3}{4}\rangle\oplus\langle-\frac{2}{3}\rangle\oplus\langle-\frac{1}{2}\rangle\oplus\langle\frac{1}{4}\rangle \cong (\Z /8\Z)\oplus (\Z /4\Z)^2\oplus (\Z /2\Z)^2\oplus(\Z /3\Z)^2.$$

We determine all isotropic subgroups $\K\subset \discr S_h$ such that the corresponding finite index extension $\tilde{S}_h$ satisfies the conditions in Definition \ref{lattice.type}, \ie, $\tilde{S}_h$ is a lattice type extending $S_h$. Up to action of $O(S)$, we have six such isotropic subgroups $\K$ (\ie, six isomorphism classes of lattice types $\tilde{S}_h$), which are given
in the \autoref{table:kernels} using the coordinate vector notation.

\begin{table}
	\centering
	\caption{The isotropic subgroups $\K_i$}\label{table:kernels}
	\begin{tabular}{l l l } 
		\hline \addlinespace[0.11cm]
		& Generators &  \\ [0.5ex] 
		\hline\addlinespace[0.11cm]
		$\K_1$ & $[ 4, 3, 0, 0, 1, 0 ]$ & cyclic of order $2$  \\ 
		$\K_2$ & $[ 0, 3, 2, 0, 1, 2 ]$ & cyclic of order $2$ \\
		$\K_3$& $[ 4, 3, 0, 0, 1, 0 ], [ 4, 0, 2, 0, 0, 2 ]$  &  of order $4$ \\ 
		$\K_4$ & $[ 6, 0, 1, 0, 0, 3 ]$ & cyclic of order $4$  \\ 
		$\K_5$ & $[ 6, 3, 1, 0, 1, 1 ]$ & cyclic of order $4$ \\
		$\K_6$& $[ 6, 0, 1, 0, 0, 3 ], [ 4, 3, 0, 0, 1, 0 ]$  &  of order $8$ \\ [1ex] 
		\hline
	\end{tabular}
\end{table}

As the lattice type corresponding to the isotropic subgroup $\K_2$ is the one giving rise to the real component of the equisingular strata (\ie, extending to a unique homological type) which does not admit a real quartic,  from now on,  we fix the lattice type $\tilde{S}_h$ as the one associated to $\K=\K_2$. Then $\discr \tilde{S}_h$, which is given by $\K^\bot/\K$, is as follows:
$$\discr \tilde S_h\textstyle \cong (\Z /8\Z)\oplus (\Z /4\Z)\oplus (\Z /4\Z)\oplus(\Z /3\Z)\oplus(\Z /3\Z).$$ 

Consider an integer $a^2\mathrel| 2\operatorname{exp} (\discr \tilde{S}_h)= 2^4\cdot3^2$ as in Step~1 in \S\ref{section.algorithm}, \emph{i.e.}, $a^2= 2^N\cdot3^r$, where $N\in\{1,2,3,4\}$ and $r\in\{0,1\}$. 

Let the cyclic group $\discr \Z a$ be generated by $\alpha:= a/a^2$ with $\alpha^2=\frac{1}{a^2}$.  With $a^2$ fixed, we are interested in  the isotropic subgroups
$$\K\subset \discr (\tilde S_h\oplus \Z a)  \cong (\Z /8\Z)\oplus (\Z /4\Z)\oplus (\Z /4\Z)\oplus (\Z /4\Z)\oplus(\Z /3\Z)\oplus(\Z /3\Z) $$
which are the cyclic subgroups $\langle \vartheta \rangle$ generated by $\vartheta=\kappa \oplus n\alpha$ such that $\operatorname{order}(\vartheta)= a^2/n$,  $n=1$ or $2$, $\kappa\in\discr \tilde S_h$, as in Lemma~5.2 in \cite{aktacs2019real}. Note that $\vartheta \cdot \alpha = n/a^2$.

With $a^2$ fixed, let the group $\discr(\tilde{S}_h\oplus \mathbb{Z}a)$ be the orthogonal sum of cyclic groups with generators $\alpha_i$, where $\alpha_0$ is the generator of $\discr \mathbb{Z}a$. The action of $O_{h}(\tilde{S}_h)\times\{\pm \id_a\}$ on $\discr (\tilde{S}_h\oplus\mathbb{Z}a)$ is generated by involutions $\alpha_i\mapsto\pm\alpha_i$.

\subsubsection{The case $N=1$}\label{subN1}
Then $n=2$, as $\discr \tilde{S}_h$ does not contain any cyclic direct summand of order $2$. Hence, we have $\ell_2(\K^{\bot}/\K)= 4$ , which implies by Theorem \ref{th.N.existence} that a primitive isometric embedding
$(\tilde{S}_{h}\oplus\mathbb{Z}a)^{\sim}\hookrightarrow \L=2\mathbf{E}_{8}\oplus 3\mathbf{U}$
 does not exist, ruling out this case by Corollary~\ref{cor:mu18} .

\subsubsection{The case $N=2$ and $n=2$}\label{subN2n2}
Then we have $\ell_2(\K^{\bot}/\K)= 4$ and hence this case ruled out as in section~\ref{subN1} .

\subsubsection{The case $N=2$ and $n=1$}\label{N2n1}
For an odd prime $p$, the $p$--primary part of the group $\langle\kappa\rangle$ is an orthogonal summand in the cyclic group $\discr_p \tilde{S_h}$, hence, in this particular case, we have either  that all the $p$--primary components of the vector $\kappa$ is zero, which implies $ p\nmid a^2$, or 
$\discr_p\tilde{S}_h$ is generated by these $p$--primary components of $\kappa$. Since 
$\discr \tilde{S}_h$ does not contain any cyclic direct summand of order $9$, we rule out the case $a^2=4\cdot3\cdot3$ and obtain $a^2=4\delta$ where $\delta=1$ or $3$.

Let $a^2=4$. To find the options for the $2$--primary part $\K_{[2]}$ of $\K$, we try to list  $v=\vartheta_{[2]} \in \discr_2 (\tilde{S}_h\oplus \Z a)$ such that
$$
\text{$\operatorname{order }(v)=4$, $\;v\cdot\alpha =\frac{1}{4\delta}$ and $v^2=0$}.
$$
If $\delta=1$, using the coordinate vector notation for the generator $v$ of the cyclic subgroup $\K_{[2]}\subset (\Z /4\Z)\oplus (\Z /4\Z)\oplus (\Z /4\Z)\oplus (\Z /8\Z)$, we obtain that $v$ is of the form:
$$
\text{$[2,0,3,\pm 2]$ or $[2,2,1,\pm 2]$ or $[3,2,2,\pm 2]$ or $[1,0,2,\pm 2]$.}
$$
For all these cases, we have  $\discr_2(\tilde{S}_{h}\oplus\mathbb{Z}a)^{\sim}=\K_{[2]}^{\bot}/\K_{[2]}\cong (\Z /8\Z)\oplus (\Z /4\Z)$ and there exist a vector in $\K_{[2]}^{\bot}/\K_{[2]}$ of order $4$ which is reversed by the map  $g\oplus (-\mathrm{id})$ with $g\in O(\tilde S_h)$.  Thus,  $g\oplus (-\mathrm{id})$ does not induce $\pm\mathrm{id}$ on $\discr_2(\tilde{S}_{h}\oplus\mathbb{Z}a)^{\sim}$, eliminating this case  by  Corollary~\ref{cor:mu18}.

\noindent If $\delta=3$, we obtain that $v$ is of the form:
$$
\text{$[0,0,1,0]$ or $[0,0,1,4]$ or $[0,2,1,0]$ or $[0,2,1,4]$,}
$$
where for all the cases we have   $\discr_2(\tilde{S}_{h}\oplus\mathbb{Z}a)^{\sim}=\K_{[2]}^{\bot}/\K_{[2]}\cong (\Z /8\Z)\oplus (\Z /4\Z)$ and there exist a vector in $\K_{[2]}^{\bot}/\K_{[2]}$ of order $4$ which is reversed by   $g\oplus (-\mathrm{id})$, eliminating this case  as in the case $\delta=1$.  
\subsubsection{The case $N=3$ and $n=1$}\label{N3n1(2)}
As in section~\ref{N2n1}, we have $a^2=8\delta$,  where $\delta=1$ or $3$. We search for  $v=\vartheta_{[2]} \in \discr_2 (\tilde{S}_h\oplus \Z a)$ such that
$$
\text{$\operatorname{order }(v)=8$, $\;v\cdot\alpha =\frac{1}{8\delta}$ and $v^2=0$}.
$$
\par If $\delta=1$, then for all such vectors $v$, by Theorem~\ref{th.N.existence}, a primitive isometric embedding
$(\tilde{S}_{h}\oplus\mathbb{Z}a)^{\sim}\hookrightarrow \L=2\mathbf{E}_{8}\oplus 3\mathbf{U}$
does not exist, ruling out this case by Corollary~\ref{cor:mu18} .

If $\delta=3$, we obtain that $v$ is of the form:
$$
\text{$[3,0,1,\pm 1]$ or $[3,2,1,\pm 3]$ or $[1,0,5,\pm 3]$ or $[2,2,5,\pm 1]$,}
$$
where for all the cases we have   $\discr_2(\tilde{S}_{h}\oplus\mathbb{Z}a)^{\sim}=\K_{[2]}^{\bot}/\K_{[2]}\cong (\Z /4\Z)\oplus (\Z /4\Z)$ and there exist a vector in $\K_{[2]}^{\bot}/\K_{[2]}$ of order $4$ which is reversed by   $g\oplus (-\mathrm{id})$, ruling out this case by Corollary~\ref{cor:mu18}

\subsubsection{The case $N=3$ and $n=2$}
By section~\ref{N3n1(2)}, we have $a^2=8\delta$,  where $\delta=1$ or $3$. Then, we list all $v\in \discr_2 (\tilde{S}_h\oplus \Z a)$ such that

$$\text{$\operatorname{order }(v)=4$, $\;v\cdot\alpha =\frac{1}{4\delta}$ and $v^2=0$}.$$
For all such vectors $v$, by Theorem \ref{th.N.existence}, there is no primitive isometric embedding
$(\tilde{S}_{h}\oplus\mathbb{Z}a)^{\sim}\hookrightarrow \L=2\mathbf{E}_{8}\oplus 3\mathbf{U}$, ruling out this case by Corollary \ref{cor:mu18} .
\subsubsection{The case $N=4$}\label{sub4}
Then $n=2$, since $\discr_2 \tilde{S}_h$ does not contain any cyclic direct summand of order $16$. On the other hand any order $8$ element in $\discr_2\tilde{S}_h$ is an orthogonal direct summand which contradicts to $n=2$.

\subsection{Simple Sextics}
In this section we consider hyperelliptic projective models $f_h\colon X\rightarrow \mathbb{P}^2$ with $h^2=2$, which corresponds to  plane sextic curves.
Simple sextics form one of the most classical families of $K3$-models, and
their deformation theory has been studied extensively.  A complete description
of the equisingular strata $\mathcal{M}(S)$ of simple sextic
curves is given in~\cite{Alex2}, where all strata of simple sextics with root
lattice $S$ are enumerated and their geometry is analyzed.
For simple sextics, the search for real representatives proceeds entirely parallel to that of quartic surfaces.
The perturbation principle (Proposition~\ref{real.pert}) ensures that realness persists under $c$–invariant degenerations, while Theorem~\ref{thm:Aktas2019} provides the same lattice-theoretic criterion for constructing real models.
Thus maximizing real types yield real sextics, and their $c$–invariant perturbations remain real as well. Among the nonspecial sextics, one particularly interesting example arises:
Degtyarev and Akyol proved that the stratum
$\mathcal{M}(\mathbf{A}_{7}\oplus\mathbf{A}_{6}\oplus\mathbf{A}_{5})$
is connected and therefore real, but nonetheless contains no real sextic
curves (see Proposition~2.6 in~\cite{Alex2}).
This phenomenon is directly analogous to the exceptional real strata  in the quartic case identified in
Theorem~\ref{principal.result}.  In both settings, the equisingular stratum
is real as a complex variety but fails to contain real representatives. The lattice–-theoretic approach developed here provides a shorter
and conceptually clearer proof of the sextic example, see~\cite{aktacs2019real} for details.

Thus, the sextic case illustrates that the exceptional behavior observed for
certain quartic types is rear but not isolated: real strata without real
representatives already appear in the simplest projective models of $K3$--surfaces.  At the same time, the general algorithm of
Section~\ref{section.algorithm} successfully detects these cases and
confirms the existence of real representatives for all remaining strata.


\subsection*{Concluding remarks}

The analysis carried out in this paper establishes a complete and effective
lattice–theoretic framework for determining the (non)existence of real
representatives in the equisingular strata of projective models of $K3$–surfaces
with simple singularities. The interaction between maximizing real types, their $c$–-invariant perturbations,
and the possibility of extending a reflection in the transcendental lattice to
a global skew–autoisometry of $\L$, gives us a unified and computationally
verifiable procedure that applies uniformly to all polarizations.  In the
quartic case, this framework leads to a full classification of real strata and
identifies precisely three real types that admits no real representatives, extending and
completing the earlier results of \cite{aktacs2019real}.  The sextic example
demonstrates that this phenomenon is intrinsic to the geometry of $K3$--surfaces and appears already in the lowest nontrivial polarization.  
Altogether, these results show that the interplay between lattice theory and
real geometry provides a powerful and flexible approach to the study of
equisingular strata, capable of resolving long-standing classification
problems and illuminating the structure of real algebraic surfaces in a broad
range of settings.

\bibliographystyle{amsplain}
\bibliography{mybibliography}

\end{document}